\documentclass[a4paper,fleqn]{cas-sc}
\usepackage[numbers]{natbib}

\usepackage{latexsym,amsmath,amsfonts,amsthm,amssymb,bbm,amsbsy}
\usepackage{mathrsfs}
\usepackage{verbatim}
\usepackage{amscd}
\usepackage{commath}
\usepackage{epsfig}
\usepackage{graphicx}
\usepackage{caption}
\usepackage[all]{hypcap}
\usepackage{enumitem}
\usepackage{mathtools}
\usepackage{booktabs}
\usepackage[table]{xcolor}
\usepackage{soul}

\usepackage{tikz}
\usetikzlibrary{decorations.pathreplacing,intersections,backgrounds}
\usepackage{pgfplots}
\usepackage{tkz-euclide}

\newcommand{\origin}{\mathbf{o}} 
\newcommand{\Lpq}{\mathcal{L}_{p,q}}
\definecolor{dgreen}{rgb}{0,0.6,0}

\newtheorem{defn}{Definition}[section]
\newtheorem{thm}[defn]{Theorem}
\newtheorem{theorem}[defn]{Theorem}

\newtheorem{lem}[defn]{Lemma}
\newtheorem{remark}[defn]{Remark}

\begin{document}

\shorttitle{Cheeger constants on hyperbolic lattices}    

\shortauthors{M. D'Achille, V. Jacquier, W.M. Ruszel}  

\title [mode = title]{On minimal shapes and isoperimetric constants in hyperbolic lattices}

\author[1]{Matteo D'Achille}[orcid=0000-0002-8750-1275]

\fnmark[1]

\ead{matteo.d-achille@univ-lorraine.fr}

\affiliation[1]{organization={Institut Élie Cartan de Lorraine},
            city={Metz},
            postcode={F-57070}, 
            country={France}}
            
\author[2]{Vanessa Jacquier}[orcid=0000-0002-0125-4825]

\cormark[1]

\fnmark[2]

\ead{vanessa.jacquier@unipd.it}

\affiliation[2]{organization={University of Padova},
            addressline={via Trieste 63}, 
            city={Padova (Pd)},
            postcode={35131}, 
            country={Italy}}

\author[3]{Wioletta M. Ruszel}[orcid=0000-0002-8166-2318]

\fnmark[3]

\ead{w.m.ruszel@uu.nl}


\affiliation[3]{organization={Utrecht University},
            addressline={Budapestlaan 6}, 
            city={Utrecht},
            postcode={3584 CD}, 
            country={The Netherlands}}

\cortext[cor1]{Corresponding author}

\begin{abstract}
 We fully characterize the set of finite shapes with minimal perimeter on hyperbolic lattices given by regular tilings of the hyperbolic plane whose tiles are regular $p$-gons meeting at vertices of degree $q$, with $1/p+1/q<\frac{1}{2}$. In particular, we prove that the ratio between the perimeter and the area (i.e., the number of vertices) of this set of minimal shapes converges to the isoperimetric constant computed in H\"aggstr\"om-Jonasson-Lyons. In fact, our regular balls which are constructed via layers and not combinatorial balls, will realize the isoperimetric constant for any fixed number of vertices.
\end{abstract}


\maketitle

\section{Introduction}

Hyperbolic lattices play a crucial role in geometry, topology, and mathematical physics when considering spaces beyond the Euclidean setting, for example by including negative constant curvature in the space.
Some concrete applications can be found in  e.g.~crystallography~\cite{osti_1979736}, non-Euclidean analog of the quantum spin Hall effect \cite{cryst} or quantum electrodynamics \cite{electro}, with remarkable experimental consequences \cite{chen23}.

 These lattices are discrete symmetry groups acting properly discontinuously on the hyperbolic plane, forming regular tilings or tessellations of the two-dimensional space with constant Ricci curvature of $-1$; for instance, such tilings can be visualized in the Poincaré disc model.  Moreover, they form simple examples of regular lattices (faces with $p\geq 3$ sides where each vertex has degree $q\geq 3$) associated to a non-Euclidean geometric setting and have been constructed through a layer decomposition in \cite{RNO}. For examples, see Figures \ref{fig:L73} and \ref{fig:L37}. We will refer to this hyperbolic lattice as $\mathcal{L}_{p,q}$.

    \begin{minipage}[t]{0.45\textwidth}
        \centering
        \includegraphics[width=0.64\textwidth]{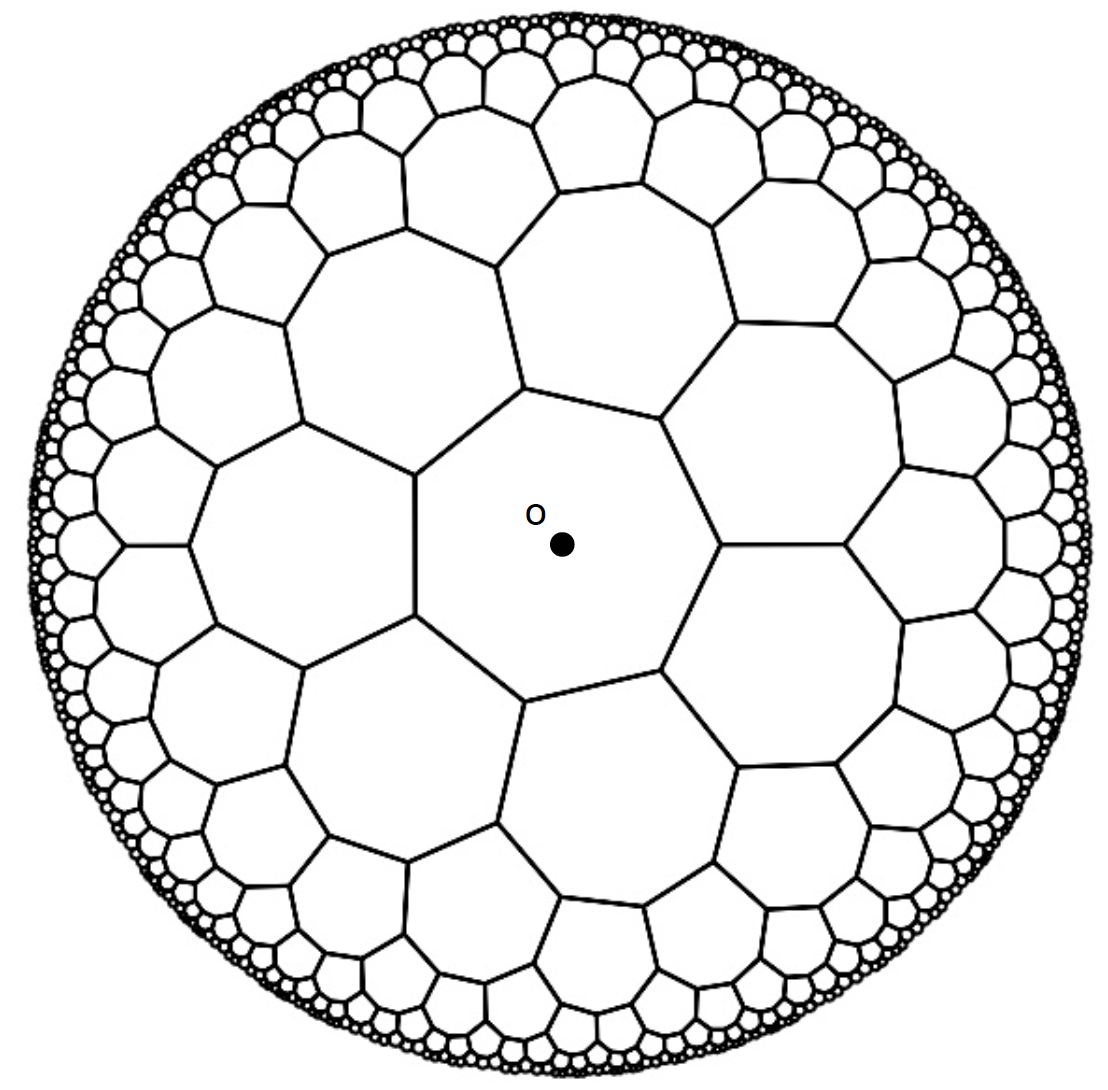} 
        \captionof{figure}{Embedding of $\mathcal{L}_{7,3}$ in the hyperbolic disc.}\label{fig:L73}
    \end{minipage}
    \hfill
\begin{minipage}[t]{0.45\textwidth}
        \centering
        \includegraphics[width=0.64\textwidth]{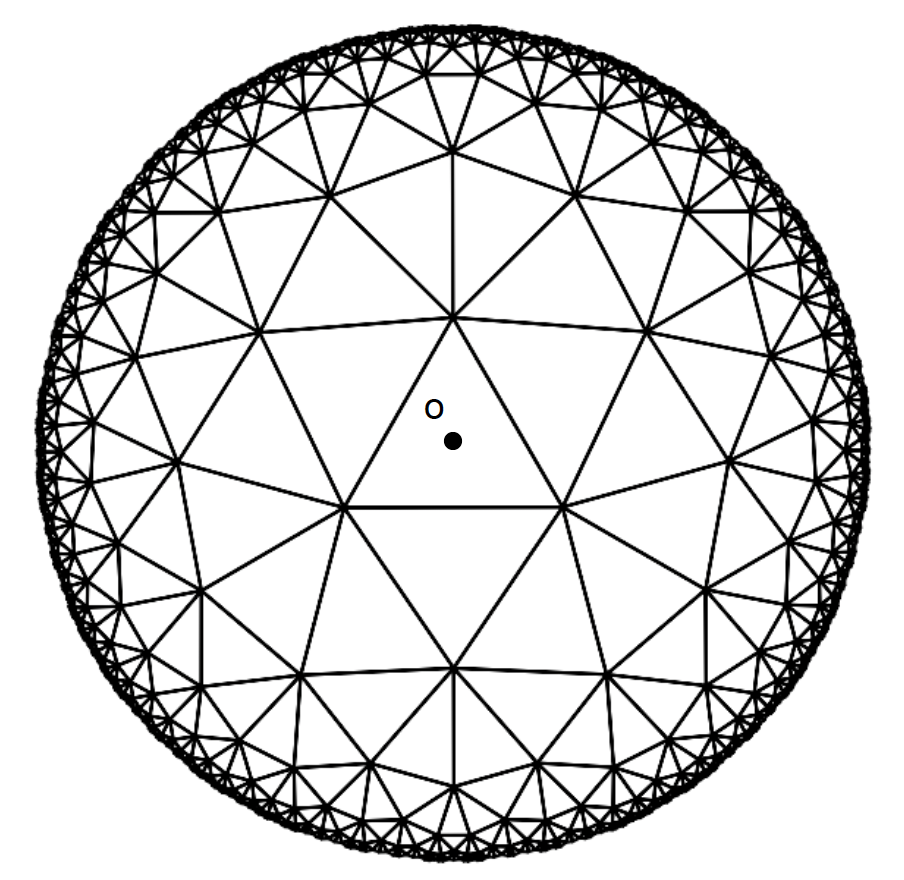} 
        \captionof{figure}{Embedding of $\mathcal{L}_{3,7}$ in the hyperbolic disc}.\label{fig:L37}
    \end{minipage}
In \cite{RNO}, the authors observed that $\mathcal{L}_{p,q}$ can alternatively be constructed through successive layers of tilings, starting from a fundamental tiling, where the vertices in each layer follow a simple recursive pattern. This recursive structure was then explicitly used to compute the growth rate of these lattices, see e.g.~\cite{moran1997growth, keller2008geometric, malen2023extremal}. Moreover, the recursive structure is useful in \cite{malen2023extremal, madras}  to obtain interesting results about the number of animals and extremal $p,q$-animals on those lattices. 

Isoperimetric (or Cheeger) constants $i_e (\cdot)$ of a graph are important because they quantify how efficiently a shape encloses an area  relative to its perimeter, and because they quantify global connectivity in a graph. Intuitively, the isoperimetric constant of a graph represents the measure of how difficult it is to separate the graph into two parts. 
For a finite, connected regular graph $H$ of degree $r$, it is well-known that $(r-\lambda)/2\leq i_e (H) \leq \sqrt{2 r (r-\lambda)}$, where $\lambda$ is the first positive eigenvalue of the graph Laplacian, see e.g.~\citep[Theorem 4.11]{hoory2006}.
Conversely, if the isoperimetric constant is zero, the graph contains finite subsets with very few boundary edges, suggesting a structure that can be easily separated (such as $\mathbb{Z}^d$ which are examples of amenable graphs). 

The 1-skeleton of a hyperbolic lattice provides an example of an infinite regular non-amenable graph $G$, i.e. $i_e(G)>0$, meaning that no finite subset has a {\it small boundary} relative to its size. 
The positivity of $i_e(G)$ explains why these graphs (and models defined on these graphs) behave differently from their counterparts on classical structures like $\mathbb{Z}^d$. 

For example, percolation follows a distinct pattern on these lattices~\cite{mertensmoore2017, haggstrom2002explicit, madras}, as e.g.~there are two percolation thresholds where one has regimes in which there are 0, infinitely many infinite clusters or 1 infinite cluster, analogously to what happens for percolation on $\mathbb{T}_d \times \mathbb{Z}$, see also \cite{Newman1990PercolationI}.

The ferromagnetic nearest-neighbors Ising model on these lattices also behaves much differently from its counterpart on $\mathbb{Z}^2$ at low temperature. In the latter setting, Aizenmann~\cite{aizenman1980} and Higuchi~\cite{higuchi1979} proved that the set of \emph{extremal} Gibbs states at low temperature consists of two measures. For hyperbolic lattices, D'Achille--Coquille--Le Ny recently proved in~\cite{DCLN} that, when $p\geq 4$ and $q\geq3$, there are uncountably many extremal Gibbs states at low temperature, indexed by certain bi-infinite graph geodesics on the dual lattice (similar to Dobrushin interfaces), providing a positive answer to a broader and longstanding conjecture by Series--Sinai~\cite{SS} in this special case.

Another important question related to the isoperimetric problem is what are the minimal shapes associated with the minimal perimeter given a fixed volume and whether they realize the isoperimetric constant. For example, in $\mathbb{Z}^d$ minimal shapes are squares and in $\mathbb{R}^d$ balls. 

In \citep[Theorem 4.1]{haggstrom2002explicit} the authors compute the value of the isoperimetric constant $i_e(G)$ for the hyperbolic lattice $\mathcal{L}_{p,q}$ and mention in  \citep[Remark 4.3]{haggstrom2002explicit} that {\it the combinatorial balls} (defined in the dual lattice $\Lpq'=\mathcal{L}_{q,p}$ starting from a vertex $\origin$) { \it do not realize the isoperimetric constants}.  

In this work, we are interested in  constructing explicitly shapes which are realising the isoperimetric inequality. 
We consider regular balls $B_{r}(\origin)$ defined as the union of layers $r+1$ of the graph instead of combinatorial balls $\mathcal{B}_{r}(\origin)$, see examples in Figures \ref{fig:in1}, \ref{fig:in2} and Definition \ref{def:regular} for a precise formulation.

\vspace{0.3cm}
    \begin{minipage}{0.45\textwidth}
        \centering
        \includegraphics[width=0.64\textwidth]{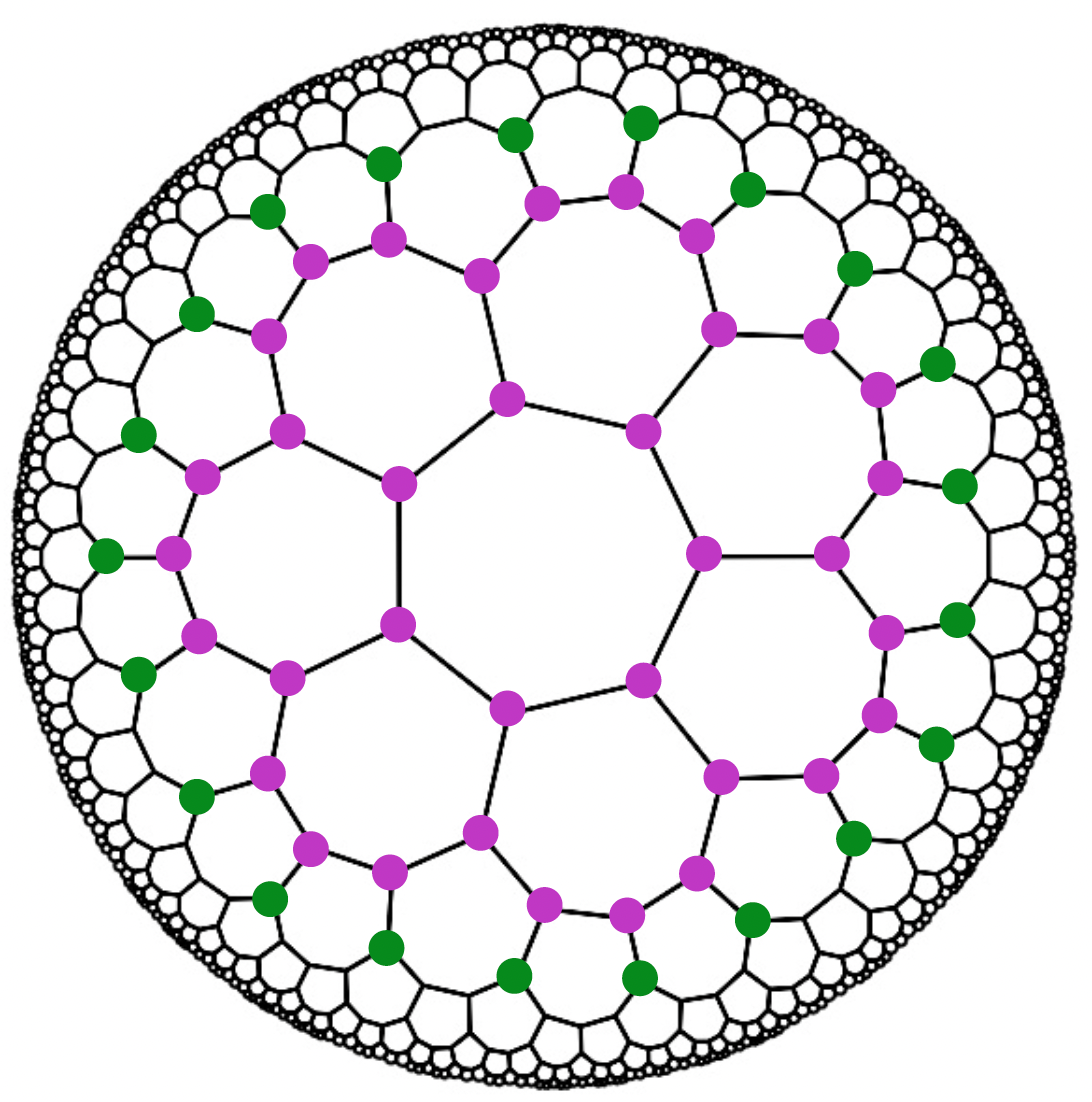} 
        \put(-73,61){$\origin$}
        \put(-69,66){$\bullet$}
        \captionof{figure}{Example of regular ball $B_1(\origin)$ in pink and the green vertices are indicating the perimeter $\partial_e B_1(\origin)$. $B_1(\origin)$ is obtained as the union of two layers, $L_0$ and $L_1$. } \label{fig:in1}
    \end{minipage}
    \hfill
    \begin{minipage}{0.45\textwidth}
        \centering
        \includegraphics[width=0.64\textwidth]{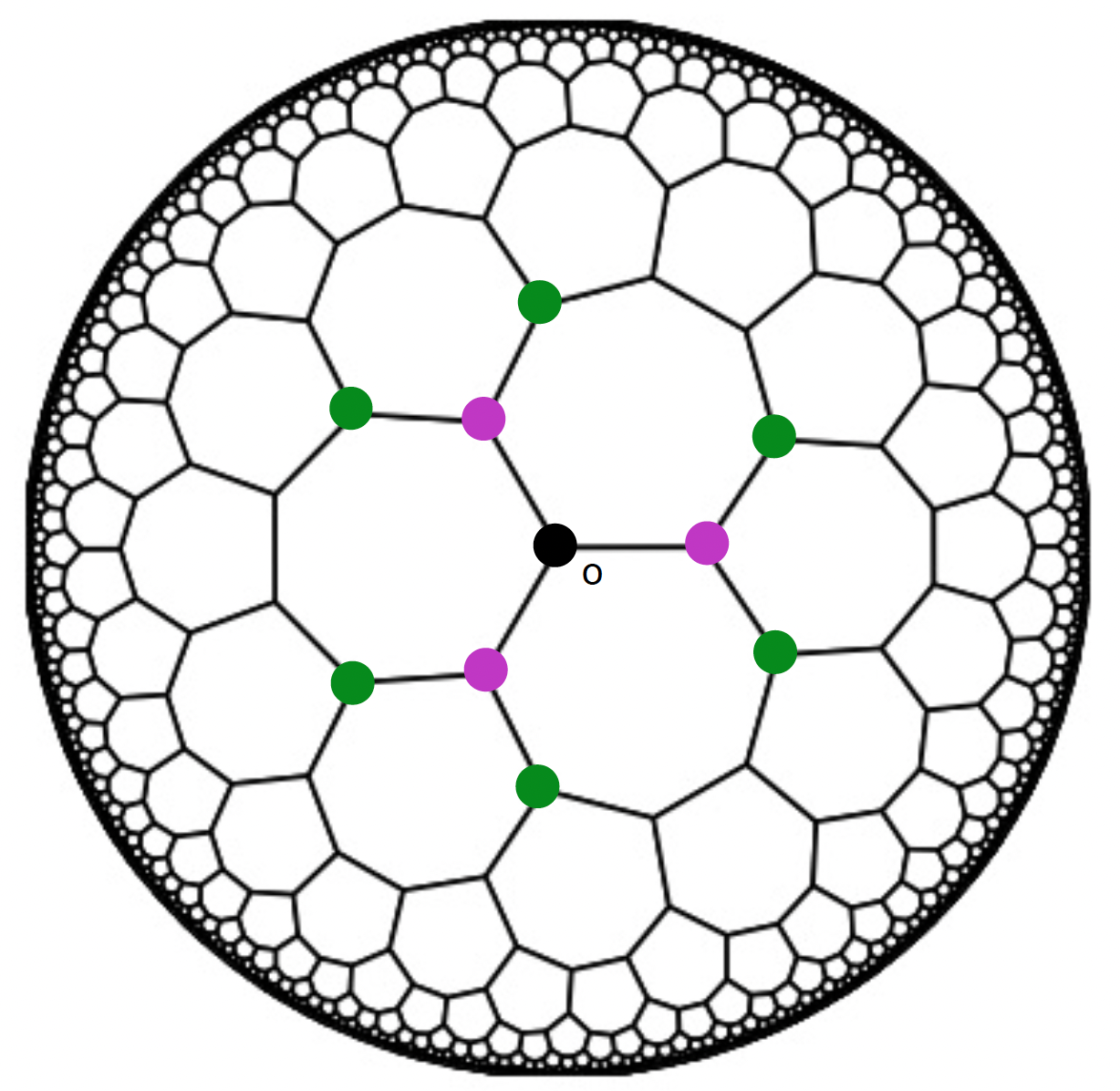} 
        \captionof{figure}{Example of a combinatorial ball $\mathcal{B}_1(\origin)$ in pink and the green vertices are indicating the perimeter $\partial_e \mathcal{B}_1(\origin)$.}\label{fig:in2}
    \end{minipage}
\vspace{0.3cm}

Our main result is the solution of the isoperimetric problem for any fixed number of vertices $n \in \mathbb{N}$.
We prove that, for every $n$, the sub-graph minimizing the perimeter among all sub-graphs of $n$ vertices is either precisely a regular ball $B_r(\origin)$ or, when 
$n$ does not coincide with the exact cardinality of such a ball, the optimal shape can be obtained by adding to the ball strips of vertices belonging to the subsequent layers in a particular way which will be defined formally later.  

This result holds for every $N\in \mathbb{N}$ and is proved directly using purely geometric and combinatorial arguments within the hyperbolic tiling. To the best of our knowledge, this is the first explicit and general proof of this type of isoperimetric behavior in the discrete hyperbolic setting. 

\begin{thm}[\textsc{Shapes of minimal perimeter}]
{\it For any fixed volume $N\in \mathbb{N}$, we characterize an explicit set of shapes with strict minimal perimeter $\mathcal{M}_N$.}
\end{thm}

\noindent 
An immediate consequence of this characterization is the computation of the Cheeger constant of the tiling, which is given by the limit as $N \to \infty$ of the ratio between the perimeter and the area (i.e., number of vertices) of the optimal region.

\begin{thm}[\textsc{H{\"a}ggstr{\"o}m--Jonasson--Lyons constant via limits of explicit minimal shapes}]
{\it Let $i_e(\mathcal{L}_{p,q})$ be the Cheeger constant computed in \citep[Theorem 4.1]{haggstrom2002explicit}, then we have 
\[
i_e(\mathcal{L}_{p,q})=\lim_{|\mathscr{M}|\rightarrow \infty} \frac{|\partial_e \mathscr{M}|}{|\mathscr{M}|},
\]
where $\mathscr{M}\in \mathcal{M}_N$.}
\end{thm}

\noindent
The resulting value coincides with that found in the existing literature for the planar regular graphs with regular dual, derived through alternative techniques, such as analytic, spectral, or probabilistic methods, but here it is recovered via a fully geometric and discrete approach. Our contribution thus offers a new combinatorial interpretation of the Cheeger constant, highlighting the deep connection between the radial structure of hyperbolic geometry and the asymptotic isoperimetric behavior.
A recent interesting work \cite{rol} in this direction considers hyperbolic formulas on animals on hyperbolic lattices. More precisely, the authors (among other results) develop formulas for the minimal perimeter of animals which consist of union of tiles rather than vertices. Their formulas are valid for all $n$-sized tile-shaped animals and contain an explicit part and some error bounds. For some specific value of $n$, which correspond geometrically to the area of our regular balls $B_r(\origin)$, the shape with minimal perimeter coincide, but in generally they do not. 

Let us further remark that in \citep[Proposition 4.6]{haggstrom2002explicit}, the authors elegantly prove that for any set which can be constructed recursively from a subset of vertices $K_0$ the ratio between the perimeter and the area of such sets converges to the Cheeger constant. Our result is thus an extension of the one presented in \cite{haggstrom2002explicit}, since it considers general minimal sets that cannot be described using the construction method outlined there, see the discussion at \ref{thm:inf} for further details.

\noindent
The structure of the paper is as follows. In Section \ref{sec:def} we will define all objects and fix notation, Section \ref{sec:res} contains all results which are proven in Section \ref{sec:proofs}.

\section{Definitions and notations}\label{sec:def}

Consider a general undirected graph $G=(V,E)$ and let $A \subset V$ be a subset of vertices. We denote by $|A|$ the cardinality of $A$ and by $A^c=V \setminus A$ the complement of $A$ in $V$. Two vertices are connected if there exists an edge between them. The set $A$ is called \textit{connected}, if for each $v,u \in A$ there exists a sequences of connected vertices in $w_1,...,w_n \in A$ such that $w_1=v$ and $w_n=u$. The {\it perimeter} $|\partial_e A|$ of the set $A$ is the cardinality of the external boundary of $A$, defined as
\begin{align}
|\partial_e A|\coloneqq |\{ (v,w) \in E \, | \, v \in A, \, w\not \in A\}|.
\end{align}

\begin{defn}\label{def:cheeger_constant}
We define the Cheeger constants for finite and infinite graphs.
\begin{itemize}
\item[(i)] (Finite graphs)
The Cheeger (or isoperimetric) constant $i_m$, resp. the Cheeger geometric constant $i^g_m$, are defined as
\begin{align}
    i_m(G)=\min_{\substack{A \subset V, \, \\ |A|=m}} \frac{|\partial_e A|}{|A|}
    \qquad \text{and} \qquad 
    i_m^g(G)=\min_{\substack{A \subset V, \, \\ |A|=m}} \frac{|\partial_e A|}{vol(A)},
\end{align}
where $vol(A)$ is the sum of the degrees of the vertices in $A$, i.e. $vol(A)=\sum_{v \in A}deg(v)$.
\item[(ii)] (Infinite graphs)
The Cheeger constant of $G$ is defined as follows
\begin{align}\label{def:isoperimetric_number}
    i_e(G)= \inf_{\substack{A \subset V, \, \\ 0<|A| < \infty}} \frac{|\partial_e A|}{|A|}.
\end{align}
\end{itemize}
\end{defn}

\noindent
We will define the hyperbolic lattice borrowing the construction from \cite{RNO}.
\noindent
Let $\{p,q\}$ be two positive integers such that $\frac{1}{p}+\frac{1}{q} <\frac{1}{2}$. $\mathcal{G}_{p,q}$ is the {\it Fuchsian group} defined as follows
\begin{equation}\label{eq:fuchs}
\mathcal{G}_{p,q} \coloneqq \langle a,b| a^q, b^p, (ab)^2\rangle
\end{equation}
where $a$ denotes the rotation around a given lattice point over an angle $\alpha=2\pi/q$ and $b$ a rotation around the center of an adjacent face over an angle $\beta =2\pi/p$. The rotation is defined w.r.t.~the hyperbolic metric (remark that we fix the scalar curvature to $-1$ with this choice)
\[
ds^2 = 4\frac{(dx^2+dy^2)}{(1-x^2-y^2)^2}.
\]
The group $\mathcal{G}_{p,q} \subset PSU(1,1)$ is a subgroup of the group of isometries of the unit disc in the complex plane. A representation $\rho$ of $a,b$ can be defined in the following way:
\[
\begin{split}
\rho(b) & =  \pm  \begin{pmatrix} e^{i\beta/2} & 0 \\ 0 & e^{-i\beta/2}\end{pmatrix} \\
\rho(a) &= \pm \frac{1}{1-r^2}\begin{pmatrix}e^{i\alpha/2}-r^2e^{-i\alpha/2} & -r(e^{i\alpha/2}-e^{-i\alpha/2} ) \\
r(e^{i\alpha/2}-e^{-i\alpha/2} ) & e^{-i\alpha/2}-r^2e^{i\alpha/2} 
\end{pmatrix}
\end{split}
\]
and $r^2=\frac{\cos((\alpha+\beta)/2)}{\cos((\alpha-\beta)/2)}$.

\begin{defn}
The hyperbolic lattice $\mathcal{L}_{p,q} = (\mathcal{V}, \mathcal{E})$ is defined as follows. Choose $z=\origin$ as the center of the fundamental face and choose $z=r$ an adjacent lattice point. The vertices are generated by words $a,b$ acting on $z=r$. Edges are drawn between points as a result of the action of $g_1,g_2 \in \mathcal{G}_{p,q}$ on $r$ if $g_2=g_1a^nba^m$ for some $m,n=0,\ldots,p-1$.
\end{defn}
\noindent
Note that the faces are equilateral and the lattice $\mathcal{L}_{p,q}$ can be naturally embedded in the Poincar\'e disc; see Figures \ref{fig:L73}, \ref{fig:L37}.

In the following, we will describe an alternative way to construct $\mathcal{L}_{p,q}$ in terms of layers.
Let $k\in \mathbb{N} \cup \{0\}$, we define the $k$-th \emph{layer}, denoted by $L_k$, as the set of vertices in $\mathcal{V}$ constructed in the following way. 

The zero layer $L_0$ is the set of $p$ vertices in the unique face of $\Lpq$ containing $\origin$. 
The first layer $L_1$ is the set of vertices (not in $L_0$) of all the faces which are adjacent to the face containing $\origin$ (including those sharing just a vertex). For $k \geq 2$, we define $L_k$ iteratively as the set of vertices 
of all the faces which are adjacent to the face containing the vertices in $L_{k-1}$. 
\begin{defn}\label{def:regular}
    We define the \emph{regular ball} of radius $k$ centered at $\origin$ as
    \begin{align}
        B_k(\origin)\coloneqq \bigcup_{l=0}^k L_l
    \end{align}
    See Figure \ref{fig:layers} for an example. 
\end{defn}
Remark that, for all $k\geq 2$, $B_k(\origin)$ cannot be simply written as the set of vertices within some fixed graph distance from a given vertex. If the layer depends on a different reference point $x \not \equiv \origin$ (such as the middle point of a tile), we will write $L_k(x)$.

    \begin{minipage}{0.45\textwidth}
        \centering
        \includegraphics[width=0.64\textwidth]{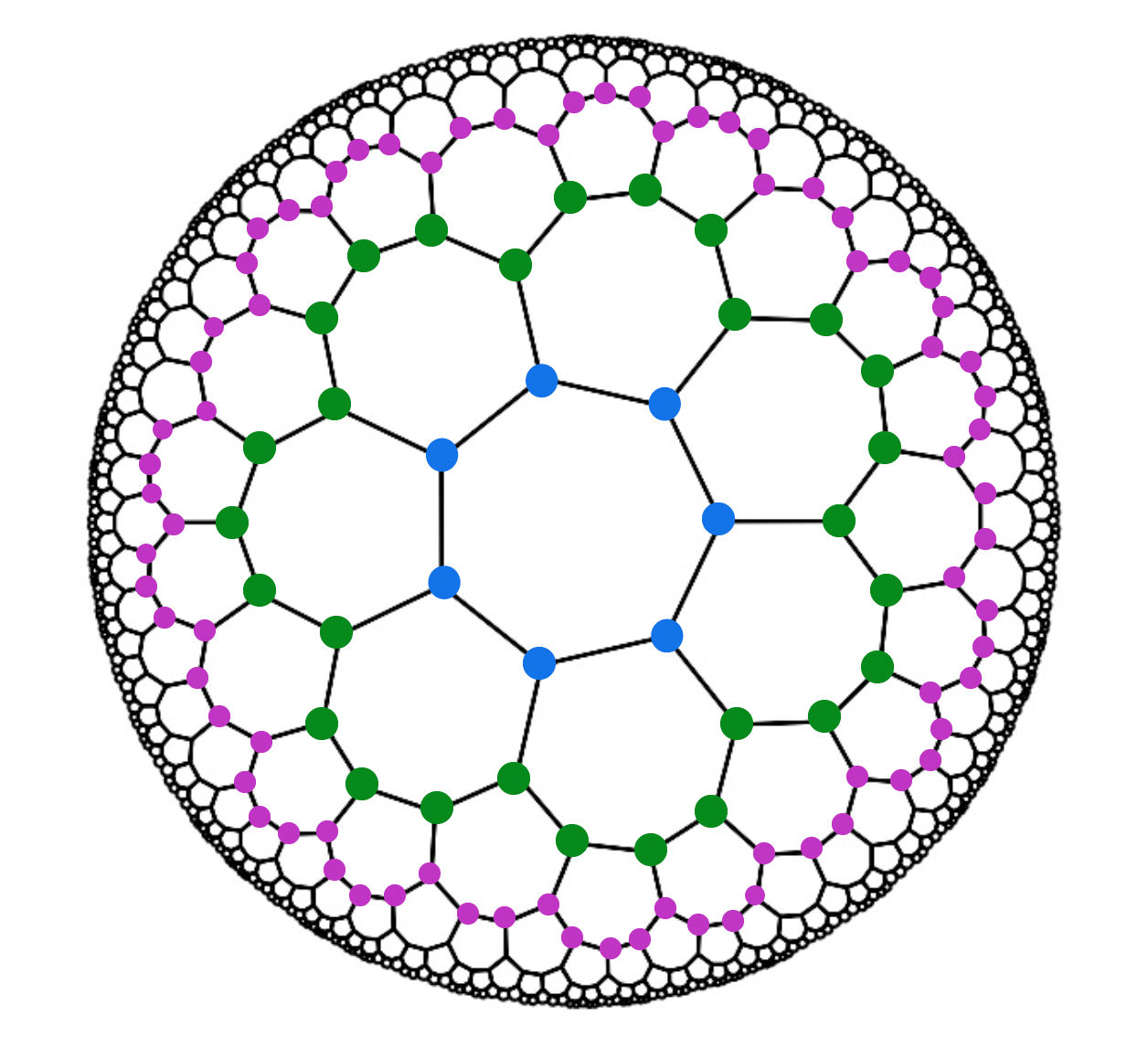} 
        \put(-73,56){$\origin$}
        \put(-69,61){$\bullet$}
       \captionof{figure}{Example of layers $L_0$ (\textcolor{blue}{blue}), $L_1$ (\textcolor{dgreen}{green}), $L_2$ (\textcolor{magenta}{pink}). Their union is the regular ball $B_2(\origin)$.}\label{fig:layers}
\end{minipage}\hfill
    \begin{minipage}{0.45\textwidth}
        \centering
        \includegraphics[width=0.58\textwidth]{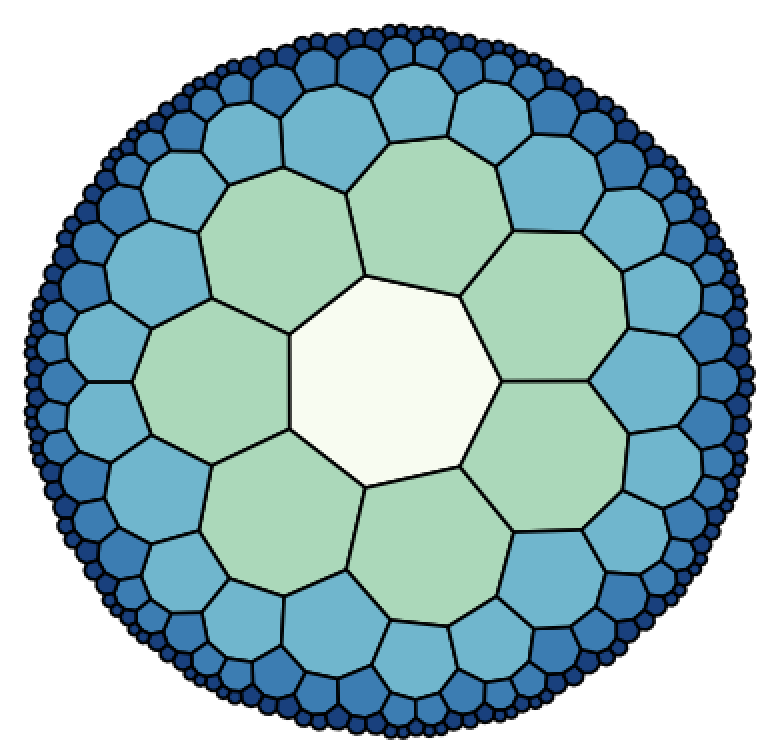} 
        \captionof{figure}{Example of balls defined via tilings in \cite{plotnick1987growth} in shades of {\color{blue}blue}.}.\label{fig:comb_ball}
    \end{minipage}

\vspace{0.3cm}
\noindent
Clearly, $\mathcal{V}=\cup_{k\geq 0} L_k$ and $L_j \cap L_k = \emptyset $ if $j\neq k$. 
Let $n$ be a fixed non-negative integer and consider first $p\geq 4$. We will distinguish between two types of vertices in each $L_n$ in a recursive way, those which are connected to vertices from the previous layer $L_{n-1}$, denoted by $I_{n;p,q}$ (from ``internal vertices'') and those which are not, denote by $E_{n;p,q}$ (from ``external vertices''), see also the construction  in \cite{RNO}. More precisely,
\begin{itemize}
    \item $I_{n;p,q} \coloneqq \{v \in L_n \, | \, \, \exists \, w \in  L_{n-1}: (v,w)\in \mathcal{E} \}$;
    \item $E_{n;p,q} \coloneqq L_n\setminus I_{n;p,q}$,
\end{itemize}
so that $|I_{n;p,q}|+|E_{n;p,q}|=|L_n|$. See Figure \ref{fig:I73} for an example. For $p=3$ (triangulations) the situation is special since all vertices in a layer are connected to the previous layer. In this case, we will define the following subsets for $n\geq 2$:
\begin{itemize}
    \item ${I}'_{n;3,q} \coloneqq \{v \in L_n \, | \, \, \exists \, w \in  L_{n-1}: (v,w)\in \mathcal{E} \}$;
    \item ${I}''_{n;3,q} \coloneqq \{v \in L_n \, | \, \, \exists \, w_1, w_2 \in  L_{n-1}, \, w_1 \neq w_2: (v,w_1), (v,w_2) \in \mathcal{E} \}$,
\end{itemize}
see Figure \ref{fig:I37}.
In this case, $|I'_{n;3,q}|+2|I''_{n;3,q}|=|L_n|$. The first set consists of the 3 points in the first triangle $L_0$. 

\noindent We will add the dependence on a reference point $x$ in the definition of the sets $I_{n;p,q}(x),$ $E_{n;p,q}(x),$ $ I'_{n;3,q}(x)$ resp.~$I''_{n;3,q}(x)$ if $x\neq \origin$.

    \begin{minipage}{0.45\textwidth}
        \centering
         \includegraphics[width=0.64\textwidth]{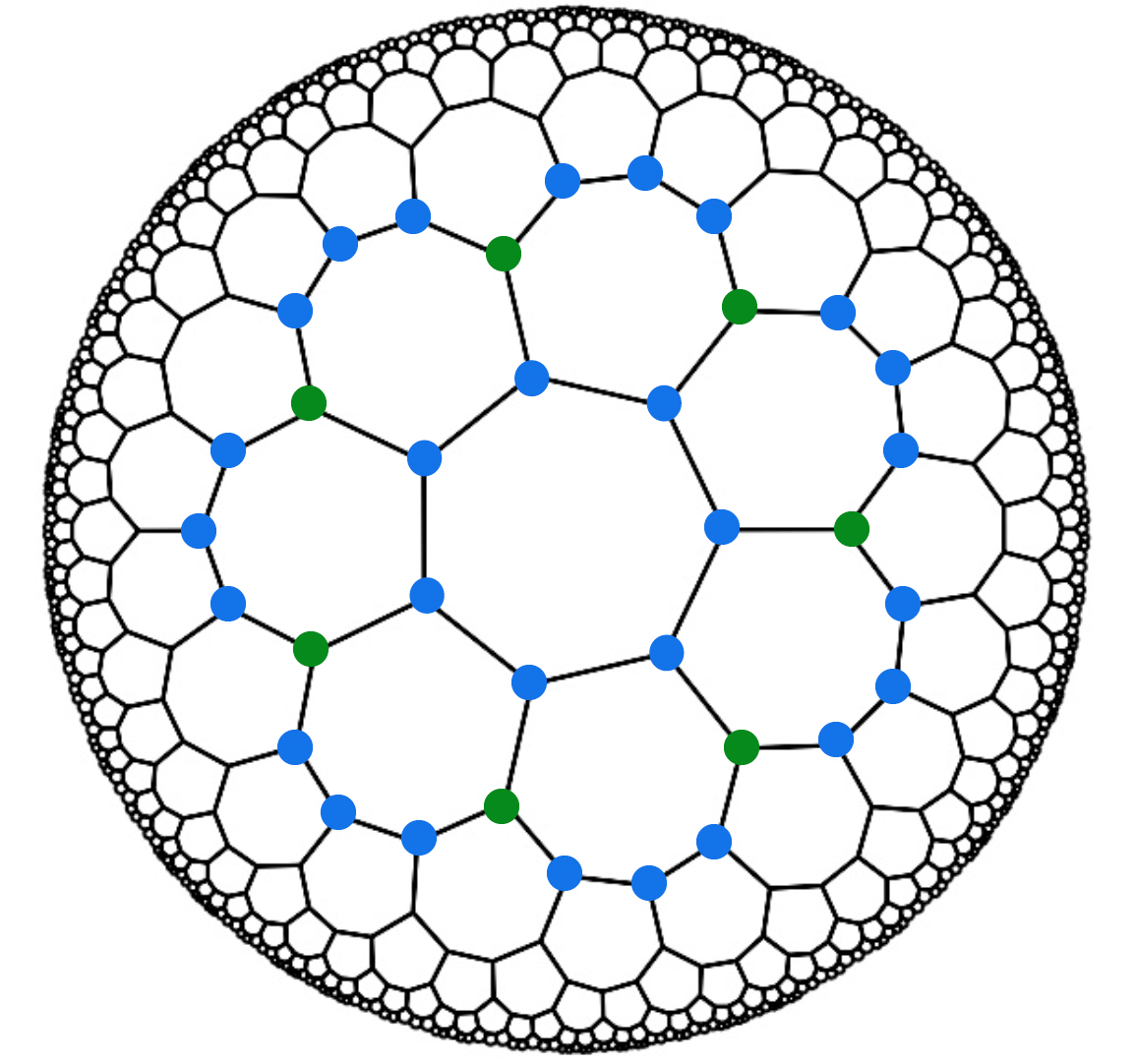} 
         \put(-73,55){$\origin$}
         \put(-70,61){$\bullet$}
        \captionof{figure}{Example of the sets $E_{0;7,3}$ and $E_{1;7,3}$ in \textcolor{blue}{blue} and $I_{1;7,3}$ in \textcolor{dgreen}{green}.} \label{fig:I73}
    \end{minipage}\hfill
    \begin{minipage}{0.45\textwidth}
        \centering
        \includegraphics[width=0.61\textwidth]{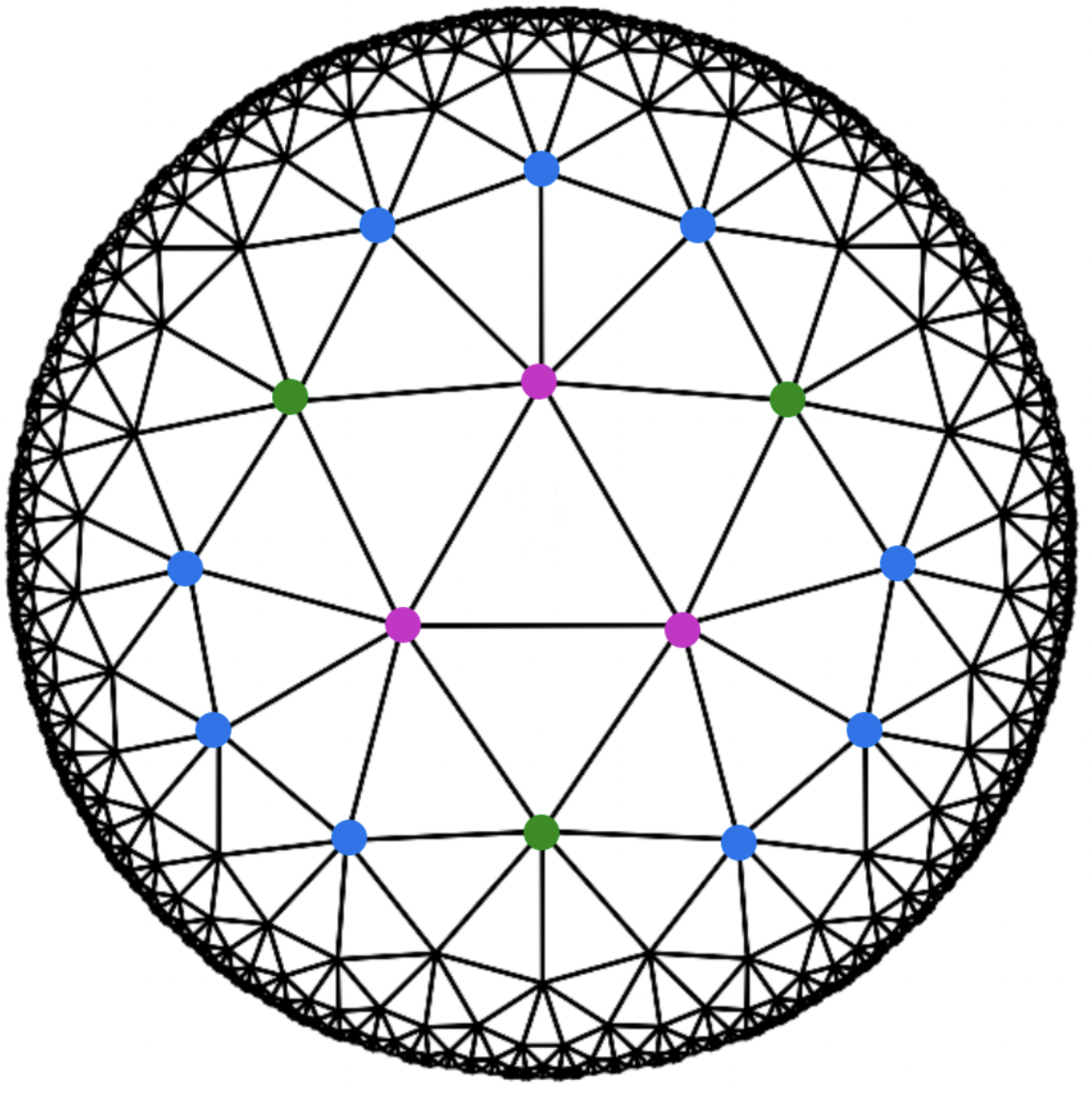} %
        \put(-67,68){$\origin$}
         \put(-67,62){$\bullet$}
        \captionof{figure}{Example of the sets $I'_{1;3,7}$ in \textcolor{blue}{blue}, $I''_{1;3,7}$ in \textcolor{dgreen}{green} and the set $L_0$ in \textcolor{magenta}{pink}. }\label{fig:I37}
    \end{minipage}

\begin{remark}
    Note that the ``combinatorial balls" $\mathcal{B}_n$ constructed in \cite{plotnick1987growth} using the word norm on the Fuchsian group $\mathcal{G}_{p,q}$ are unions of tiles. If we look at the corresponding combinatorial ball in the dual lattice $\mathcal{B}'_n$ of radius $n$ from the origin $\origin$ defined in \cite{haggstrom2002explicit}, then $|\partial_e \mathcal{B}'_n| < |\partial_e B_{n}|$ and $|\partial_e \mathcal{B}'_n|/|\mathcal{B}'_n| > |\partial_e B_{n}|/|B_{n}|$.  This is the reason why their combinatorial balls, which look ``spiky'', will not realize the isoperimetric constant (see Remark 4.3 in \cite{haggstrom2002explicit}) and ours, which look ``roundy'', will. For  an example comparison we refer to Figures \ref{fig:ball_h} vs \ref{fig:ball_h2}. 
\end{remark}

        \begin{minipage}{0.45\textwidth}
        \centering
        \includegraphics[width=0.64\textwidth]{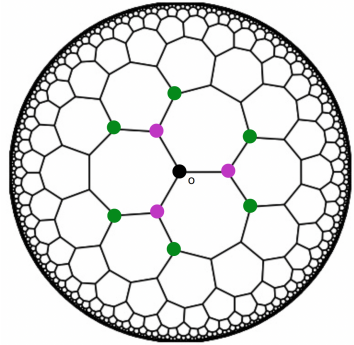} 
        \captionof{figure}{Example of $\mathcal{B}_1(\origin)$ in pink and the green vertices are indicating $\partial_e \mathcal{B}_1(\origin)$. We have  $|\partial_e \mathcal{B}_1(\origin)|/|\mathcal{B}_1(\origin)|= 2$.}\label{fig:ball_h2}
    \end{minipage}
    \hfill
    \begin{minipage}{0.45\textwidth}
        \centering
        \includegraphics[width=0.64\textwidth]{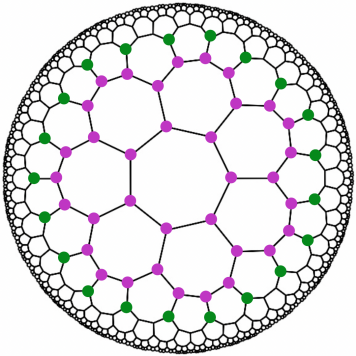} 
        \put(-74,60){$\origin$}
        \put(-69,65){$\bullet$}
        \captionof{figure}{Example of $B_1(\origin)$ in pink and the green vertices are indicating $\partial_e B_1(\origin)$. We have  $|\partial_e B_1(\origin)|/|B_1(\origin)|= 3/5$.} \label{fig:ball_h}
    \end{minipage}

\vspace{0.3cm}
\noindent
For any connected set of vertices $A$, we will define the smallest regular ball containing $A$ and the largest ball contained in $A$. 

\begin{defn}\label{def:bmax}
Let $A \subset \mathcal{V}$ be any connected set of vertices (order them lexicographically $\preceq$), and set $|A|=N$. We define $B_{A, max}$ the largest ball contained in $A$ as follows. If the vertices in $A$ do not form a polygon (or a tile) then set $B_{A,max}=\emptyset$. Otherwise, there exist $M \leq N-p$ layers $L_1(x_i)$ in $A$ with middle points $x_1,\ldots,x_M$. Let
\[
\{\overline{x}, \overline{m}\} \coloneqq \underset{l\in \mathbb{N}}{arg\,max} \, \underset{\{x_1,\ldots,x_M\}}{arg\, max} \left\{\bigcup_{k=0}^{l} L_k(x_i): \bigcup_{k=0}^{l} L_k(x_i) \subset A \right\}.
\]
Then $B_{A, max}$ is defined as
\[
B_{A, max}(\overline{x}) \coloneqq  \bigcup_{k=0}^{\overline{m}} L_k(\overline{x}).
\]
\end{defn}
\noindent Note that this ball is not uniquely defined. In the case that several sets $B_{A,max}(\overline{x})$ can be constructed in this way we take the regular ball with the smallest reference point in lexicographic order.

\begin{defn}\label{def:bmin}
Given $\overline{x}$ as in Definition \ref{def:bmax}, 
let 
\[
\overline{M}= \underset{l\in \mathbb{N}}{arg\, min} \left \{  \bigcup_{k=0}^{l} L_k(\overline{x}): \bigcup_{k=0}^{l} L_k(\overline{x}) \supset A\right \}.
\]
We define the minimal ball containing $A$, by
\[
B_{A,Min}(\overline{x})=  \bigcup_{k=0}^{\overline{M}} L_k(\overline{x}).
\]
\end{defn}
\noindent
An example can be found in Figure \ref{fig:balls}.

\begin{minipage}{0.9\textwidth}
\centering
\includegraphics[scale=0.25]{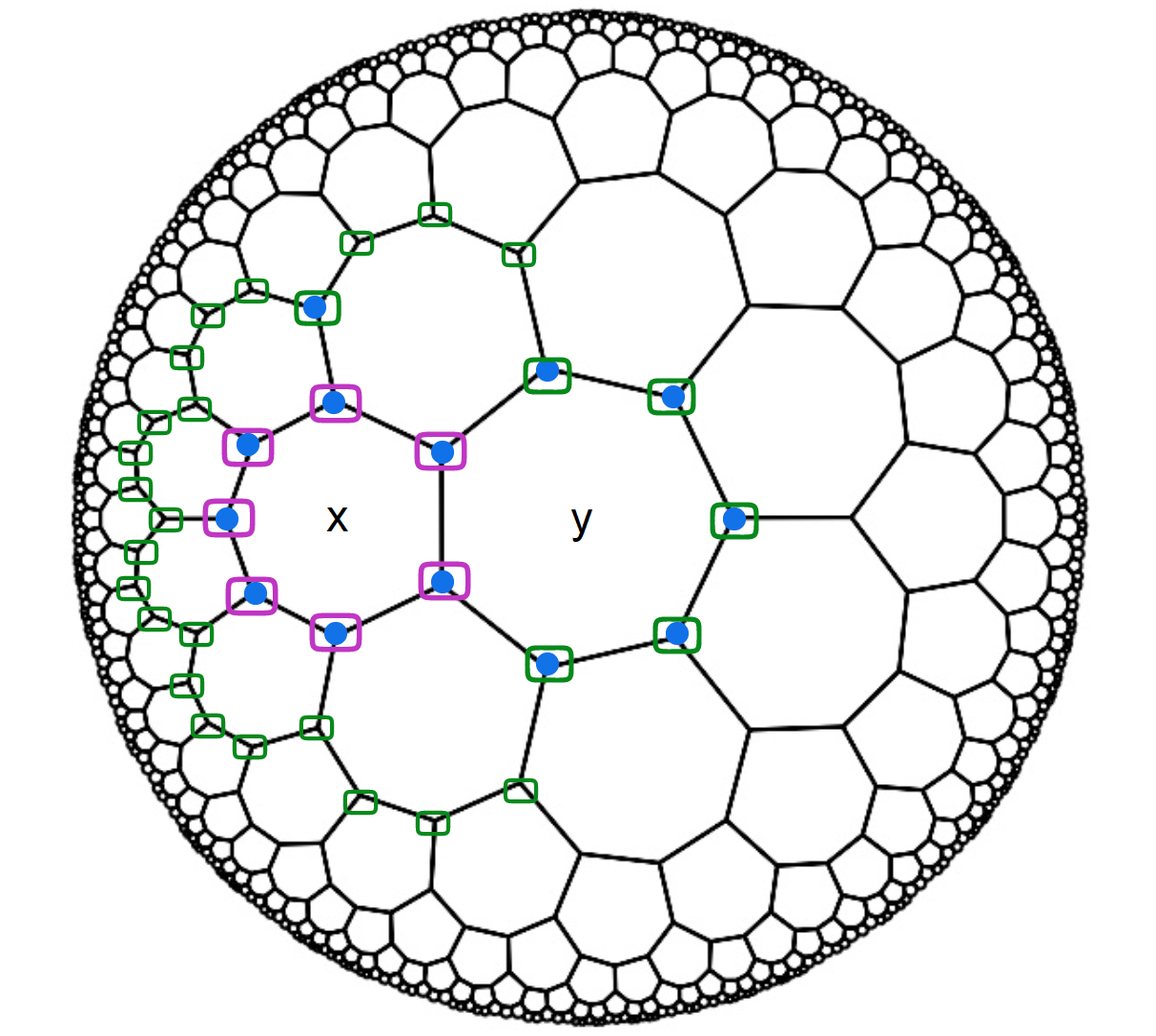}
\captionof{figure}{Example of the set of connected points $A$, displayed as \textcolor{blue}{blue} points. There are two tiles present in $A$. One is centered at $x$ and one at $y$ with $x\preceq y$. The ball $B_{A,max}(x)$ is displayed by the \textcolor{magenta}{pink} circles and and the layer $B_{A,Min}(x)\setminus B_{A,max}(x)$ by the \textcolor{dgreen}{green} circles.}\label{fig:balls}
\end{minipage}

\vspace{0.3cm}

Furthermore we will characterize the vertices $v$ in the layers in $B_{A,Min}(\overline{x})\setminus B_{A,max}(\overline{x})$ as empty if $v\notin A$ and occupied if $v\in A$. A sequence of consecutive empty, resp.~occupied, vertices in the same layer is called a \textit{strip}.
In particular, we define a peculiar strip $S$ as follows.
\begin{defn}
Let $N \in \mathbb{N}$ and and let $A$ be a connected set $A\subset \mathcal{V}$ such that $|A|=N$, $B_{A,Min}(\overline{x})\setminus B_{A,max}(\overline{x}) \subset L_{\overline{m}+1}$ and with a sequence of connected vertices in $L_{\overline{m}+1} \cap A$. We denote this sequence by $S$ and by $o_{max}$ the maximal possible number of vertices $v \in S$ which are also in $I_{\overline{m}+1;p,q}$ for $p\geq 4$, resp. also in $I''_{\overline{m}+1;3,q}$ for $p=3$.
\end{defn}

Finally, for a fixed $N \in \mathbb{N}$ and connected set $A\subset \mathcal{V}$ such that $|A|=N$ we will define sets $\mathcal{M}_N$ which will turn out to be the set of \emph{minimal shapes} which have minimal perimeter.

\begin{defn}\label{def:minset}
Fix $N\in \mathbb{N}$ and let $A$ be a connected set $A\subset \mathcal{V}$ such that $|A|=N$. We call 
\[
\mathcal{N}_e=(B_{A,Min}(\overline{x})\setminus B_{A,max}(\overline{x}))\cap A^c, \, \, \text{ resp. } \mathcal{N}_o=(B_{A,Min}(\overline{x})\setminus B_{A,max}(\overline{x}))\cap A
\]
the set of empty, resp. occupied vertices, in $B_{A,Min}(\overline{x})\setminus B_{A,max}(\overline{x})$. Let $s_e$ denote the number of empty strips in the layers of $B_{A,Min}(\overline{x})\setminus B_{A,max}(\overline{x})$.
Then $A\in \mathcal{M}_N$ if and only if it satisfies one of the following conditions:
\begin{itemize}
\item[(C1)] $s_e=0$ and $B_{A,Min}(\overline{x}) = B_{A,max}(\overline{x})$.
\item[(C2)] $s_e\geq 1$ and the set $\mathcal{N}_o$ contains precisely $o_{max}+(s_e-1)$ vertices $v$ such that $v\in \bigcup_{r=\overline{m}+1}^{\overline{M}} I_{r;p,q}$ resp. $v\in \bigcup_{r=\overline{m}+1}^{\overline{M}} I''_{r;3,q}$.
\end{itemize}
\end{defn}
We construct some examples of sets in $\mathcal{M}_{17}$ in Figures \ref{fig:min1} and \ref{fig:min2}. 

\vspace{0.3cm}
    \begin{minipage}{0.45\textwidth}
        \centering
        \includegraphics[width=0.64\textwidth]{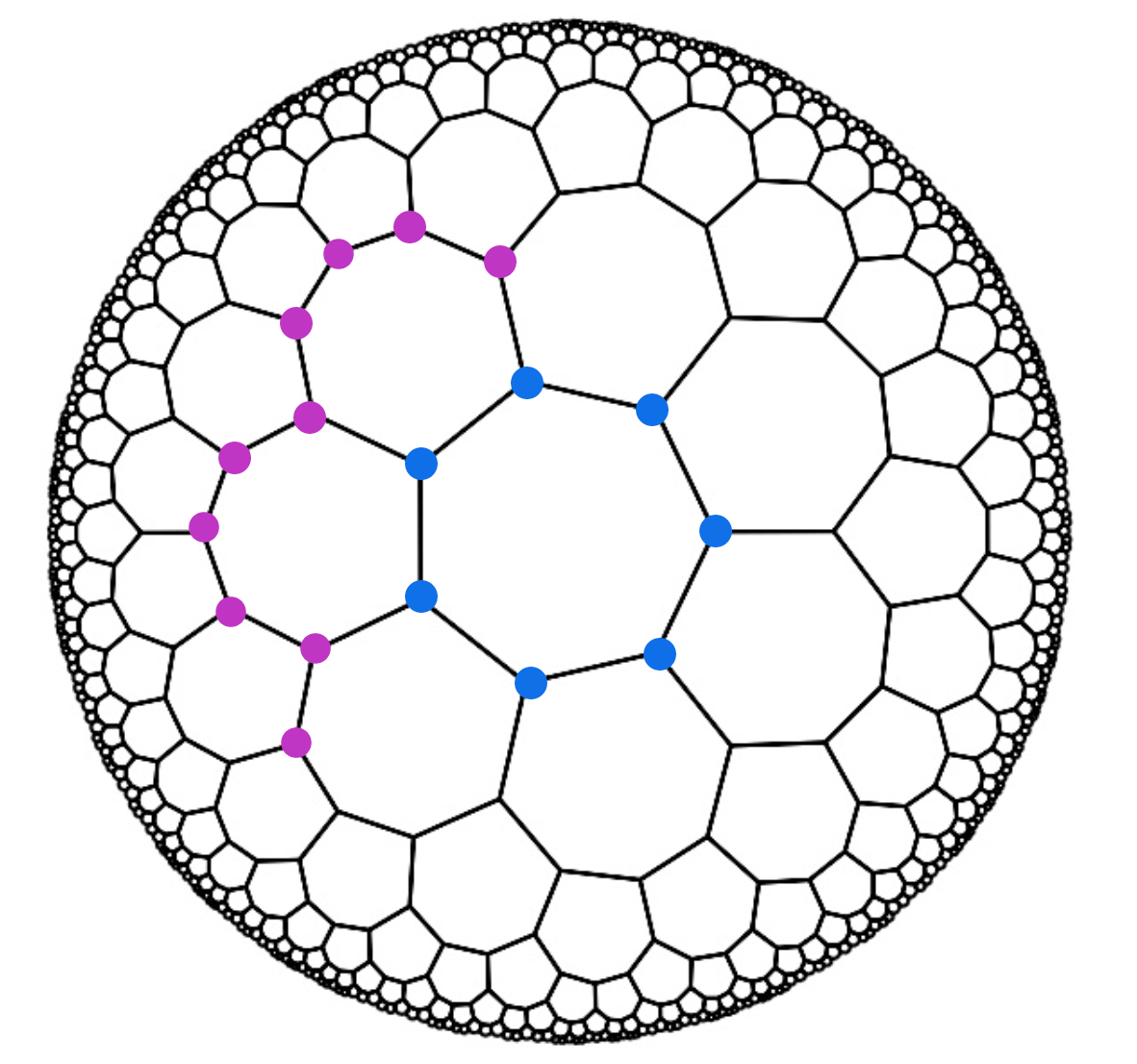} 
        \put(-75,55){$\origin$}
        \put(-70,60){$\bullet$}
        \captionof{figure}{Example of $A\in \mathcal{M}_{17}$ with $|\mathcal{N}_o|=10$, $|\mathcal{N}_e|=18$. We have $o_{max}=3$, $s_e=1$, $\overline{m}=0$, and three vertices are in $I_{1;7,3}$.}\label{fig:min1}
    \end{minipage}
    \hfill
    \begin{minipage}{0.45\textwidth}
        \centering
        \includegraphics[width=0.64\textwidth]{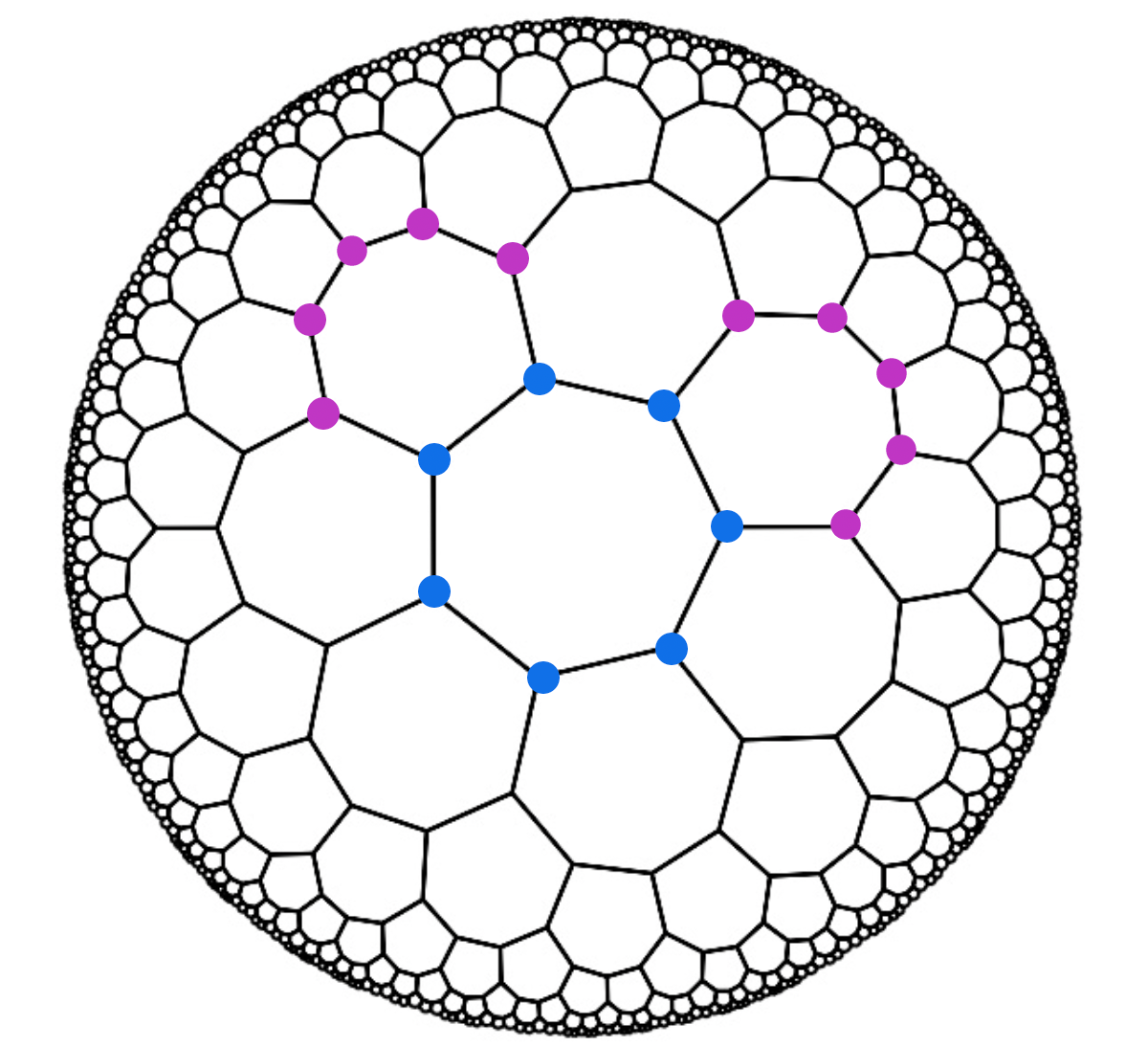} 
        \put(-75,54){$\origin$}
        \put(-70,59){$\bullet$}
        \captionof{figure}{Example of $A\in \mathcal{M}_{17}$ with $|\mathcal{N}_o|=10$, $|\mathcal{N}_e|=18$. We have $o_{max}=3$, $s_e=2$, $\overline{m}=0$, and $3+(2-1)$ vertices are in $I_{1;7,3}$.}\label{fig:min2}
    \end{minipage}

\vspace{0.3cm}
We observe that regular balls satisfy this definition, so they belong to $\mathcal{M}_N$ for some $N\in \mathbb{N}$.

\section{Results}\label{sec:res}
In this section, we present the main results of this paper. 
The first result is the explicit characterization of the sets of minimal perimeter for a fixed volume $N \in \mathbb{N}$ for any hyperbolic lattice $\mathcal{L}_{p,q}$.

\begin{theorem}[Minimal perimeter]\label{minimizers}
Given $N \in \mathbb{N}$, let $\mathcal{S}_N=\{A \subset \mathcal{V}: |A|=N\}$ be the set of all subsets of vertices with size $N$. 
Then for any $A\in \mathcal{S}_N\setminus \mathcal{M}_N$:
\[
|\partial_e A| > |\partial_e \mathscr{M}|, 
\]
where $\mathscr{M}\in \mathcal{M}_N$. Moreover, given a finite graph $G$ with vertex set $B_N(\origin)$, the Cheeger and geometric Cheeger constants are equal to
\begin{align}
i_N \left(G \right)=\frac{|\partial_e \mathscr{M}|}{N} 
    \qquad \text{and} \qquad 
i_N^g \left( G\right)=\frac{|\partial_e \mathscr{M}|}{qN},\notag
\end{align}
where $i_N, i_N^g$ are defined in Definition \ref{def:cheeger_constant}.
\end{theorem}

In the following theorem we will show that our minimal shapes in $\mathcal{M}_N$ realize the Cheeger constant (or isoperimetric constant) computed in \cite[Theorem 4.1]{haggstrom2002explicit},  when $N\rightarrow \infty$.

\begin{theorem}[Realizing the Cheeger constant]\label{thm:inf}
For some $N \in \mathbb{N}$, let $\mathscr{M}\in \mathcal{M}_N$ and let the Cheeger constant be equal to
\[
i_e(\Lpq) =(q-2) \sqrt{1-\frac{4}{(p-2)(q-2)}}.
\]
Then
\begin{align}\label{eq:iso_Lpq}
    i_e(\Lpq) &= \lim_{|\mathscr{M}| \to \infty} \frac{|\partial_e \mathscr{M}|}{|\mathscr{M}|}.
    \end{align}
\end{theorem}

Recall that our regular balls $B_n(\origin)$ represent a special case of sets built using the same recursive procedure as in \cite[Proposition 4.6]{haggstrom2002explicit}, taking our initial set $L_0$ as $K_0$. 
For instance, the following set does not meet the requirements of \cite{haggstrom2002explicit}: a ball to which a strip is attached, see e.g.  Figure \ref{fig:min1} where  there exists a tile in the external layer that intersects the strip only at certain vertices of $E_{1;7,3}$ (e.g. the tile containing 2 pink and 3 blue vertices). 

\section{Proofs }\label{sec:proofs}
In this Section we will prove our main results. 
\subsection{Proof of Theorem \ref{minimizers}} 
In order to prove Theorem \ref{minimizers} we need the following lemma. In this lemma we show that $\mathcal{M}_N$ only contains connected sets. 

\begin{lem}\label{lem:not_connected}
Let $A \subset \mathcal{V}$ be a non-connected set of vertices, and $|A|=N$. Then there exists a set $B\subset \mathcal{V}$ such that $|B|=|A|$ and $|\partial_e B| <
|\partial_e A|$.
\end{lem}
\begin{proof}[Proof of Lemma \ref{lem:not_connected}]
        Suppose that $A$ is composed of $m$ connected components $A_1,...,A_m$ with $m>1$. W.l.o.g. let $A_2$ be the closest (in graph distance) connected set of vertices to $A_1$ in the same layer. Let $B$ be the set of vertices composed of $m-1$ connected components $B_1,...,B_{m-1}$ such that $B_i=A_i$ for $i=3,...,m-1$ and 
        $B_1=A_1 \cup H(A_2)$ is a connected set of vertices where $H(A_2)$ is the homomorphism translating vertices from $A_2$ towards $A_1$. Thus, $|\partial_e B| \leq |\partial_e A|-1< |\partial_e A|$ and we conclude.
\end{proof}

In the following, we will prove Theorem \ref{minimizers}.
Given any set $D\in \mathcal{S}_N\setminus \mathcal{M}_N$ we will construct a set $A$ out of $D$ by elementary operations such that $|\partial_e D| > |\partial_e A|$. By Lemma \ref{lem:not_connected} we can exclude all sets $D$ which are not connected. Given a connected set $D$, we construct the balls $B_{D,Min}(\overline{x})$ (resp. $B_{D,max}(\overline{x})$) containing (resp. contained in $D$), see also Definitions \ref{def:bmax} (resp. \ref{def:bmin}). Note that we can write any such $D$ as 
\[
D = B_{D,max}\cup \mathcal{N}_o.
\]
To ease notation we will assume w.l.o.g.~that $\overline{x}=\origin$. The strategy will be to show that any arrangement of occupied/empty strips of vertices in the annulus $B_{D,Min} \setminus B_{D,max}$ which is satisfying conditions (C1) or (C2) in Definition \ref{def:minset}, has larger perimeter.
We will develop the proof for the case $p\geq 4$ and leave $p=3$ for the reader since it is a simple adaptation replacing the set $I$ by $I''$.

In the following argument we will construct a set of smaller perimeter than $D$ depending on the number of empty strips $s_e$ in the annulus. Recall that if $s_e=0$, then $D\in \mathcal{M}_N$ which is a contradiction.
We will distinguish between three cases:
\begin{itemize}
\item[(I)] $s_e=1$ and the empty strip is in the last layer, $S_e \in L_{\overline{M}}$ .
\item[(II)] $s_e=1$ and the empty strip is not in the last layer, $S_e \in L_{k}$ where $k<\overline{M}$.
\item[(III)] Several empty strips, $s_e > 1$.
\end{itemize}

\paragraph{\texttt{CASE (I)}}
Note that in this case $B_{D,Min} = B_{D, max} \cup L_{\overline{M}}$.
Call $S_o$ the occupied strip in $L_{\overline{M}}$, in this case we have that $S_o=\mathcal{N}_o$. Then we know that its vertices belong either to $I_{\overline{M};p,q}$ resp. to $E_{\overline{M};p,q}$. In fact, we can decompose a general strip $S_o$, as
\begin{equation}\label{eq:S_decomp}
S_o = (S_o \cap I_{\overline{M};p,q} ) \cup (S_o\cap E_{\overline{M};p,q}) \coloneqq  S_o(I)\cup S_o(E), 
\end{equation}
and $|S_o(I)| < o_{max}$ by assumption, see also Figure \ref{fig:thm1} for an example. If $|S_o(I)|=o_{max}$, then $D\in \mathcal{M}_N$ which is a contradiction. Recall that $o_{max}$  was the maximal number of vertices in $I_{\overline{M};p,q}$  an occupied strip of size $|\mathcal{N}_o|$ can have. By a direct computation, we obtain 
\begin{equation}\label{eq:B'}
|\partial_e B_{D,max}| \leq |\partial_e D| +(|S_o(E)|+2)(4-q)+(|S_o(I)|-2)(6-q). 
\end{equation}
We will construct $A$ as follows, an example can be found in Figure \ref{fig:thm2}. Let $S'_o$ denote an occupied strip in layer $L_{\overline{M}}$ such that $|S'_o|=|S_o|$ and $|S'_o(I)|=o_{max}$. Since $|S_o(I)| < o_{max}$, we then necessarily have that $|S_o(E)| > |S'_o(E)|$ and 
\begin{equation}\label{eq:omax}
o_{max} = |S_o(I)|+|S_o(E)| - |S'_o(E)|.
\end{equation}
Set now $A\coloneqq B_{D,max}\cup S'_o$. Note that the construction of the set $A$ is not unique. Therefore,
    \begin{align}
    |\partial_e A| &= | \partial_e B_{D,max}|+(|S'_o(E)|+2)(q-4)+(m-2)(q-6) \notag \\
    &\overset{\eqref{eq:B'}}\leq |\partial_e D| + (|S'_o(E)|+2)|(q-4) - (|S_o(E)|+2)(q-4)+(|m|-2)(q-6) - (|S_o(I)|-2)(q-6)\notag \\
    &= |\partial_e D|-(|S_o(E)|-|S'_o(E)|)(q-4) + (o_{max}-|S_o(I)|)(q-6) \notag \\
    & \overset{\eqref{eq:omax}}= |\partial_e D|-2\underbrace{(o_{max}-|S_o(I)|)}_{> 0 \text{ by assumption}}< |\partial_e D|.
    \end{align}

    \begin{minipage}{0.45\textwidth}
        \centering
        \includegraphics[width=0.64\textwidth]{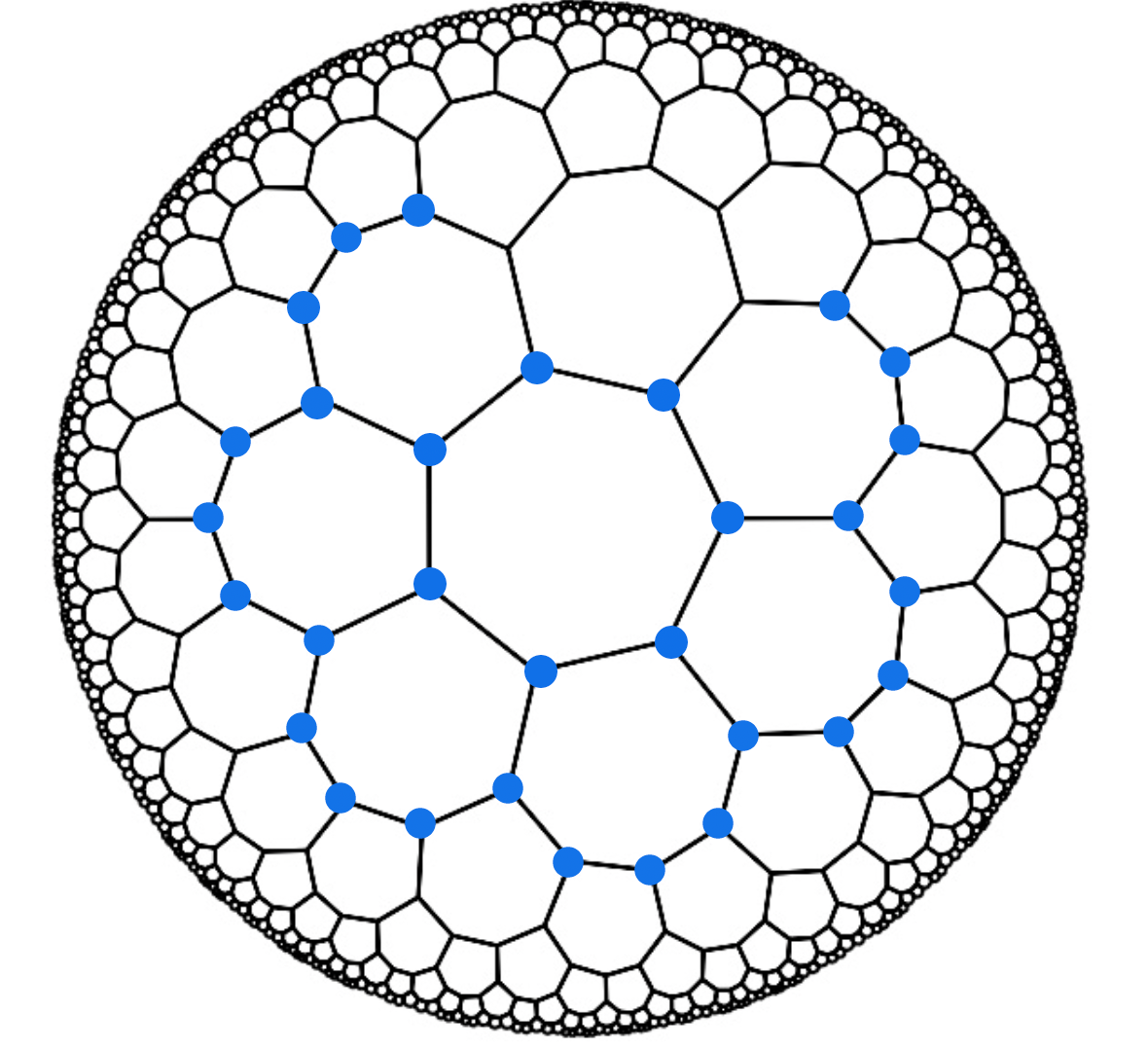} 
         \put(-76,55){$\origin$}
        \put(-73,60){ $\bullet$}
        \captionof{figure}{Example of $D\in \mathcal{S}_{30}$ with one empty and one occupied strip in the last layer, $|S_o(I)|=5$.} \label{fig:thm1}
    \end{minipage}\hfill
    \begin{minipage}{0.45\textwidth}
        \centering
        \includegraphics[width=0.64\textwidth]{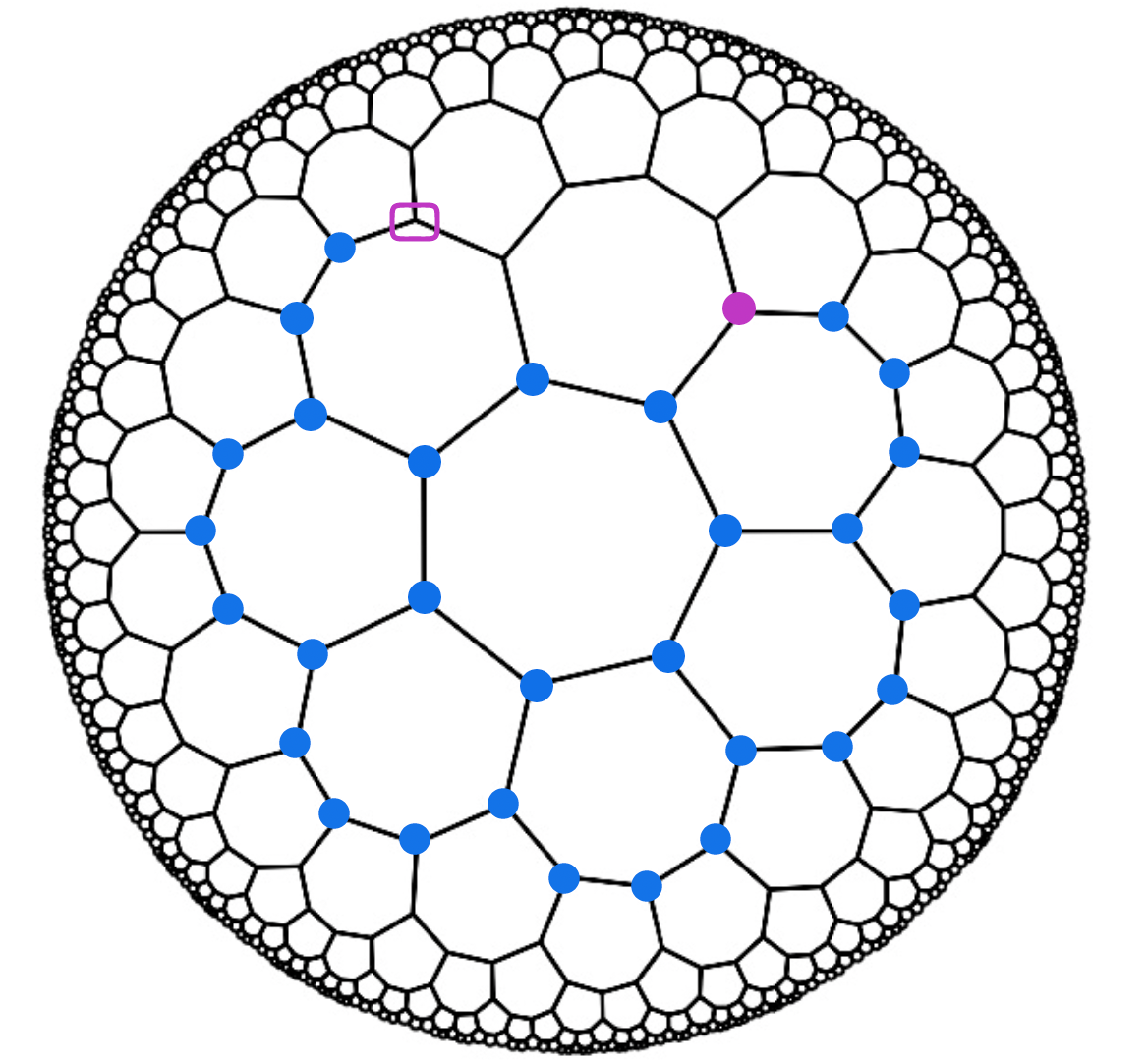} 
         \put(-76,55){$\origin$}
        \put(-73,60){ $\bullet$}
        \captionof{figure}{Example of $A$ constructed from $D$ by defining a new strip $S'$ such that $|S'_o(I)|=o_{max}=6$. }\label{fig:thm2}
    \end{minipage}

\paragraph{\texttt{CASE (II)}}
In this case we have one empty strip $S_e$ lying in $L_k$ for some $k<\overline{M}$. If $q\geq 4$ we can proceed as follows. 
First note that for a set $A'\coloneqq D\cup (S_e\setminus \{v\})$, where we added a vertex to the empty strip, the perimeter satisfies
\begin{equation}\label{eq:case2}
|\partial_e A'| \leq |\partial_e D| + (2-q).
\end{equation}
Construct $A$  by removing an occupied vertex $w$ from the last layer in $E_{\overline{M};p,q}$, i.e. $A=A'\cup\{w\}$. Then $|A|=|D|$ and 
\[
|\partial_e A| = |\partial_e A'| + (4-q)
\]
and using Equation \eqref{eq:case2} and $q\geq 4$ we obtain
\[
|\partial_e A| \leq |\partial_e D| + 6 -2q < |\partial_e D|.
\]
However, if $q=3$, moving one vertex will not be enough to get a set with strictly smaller perimeter, we instead have to fill the whole empty strip with vertices from the last layer. Note that
\[
|\partial_e B_{D, Min}| = |\partial_e D| - |S_e| -3. 
\]
Call $S'_o$ an occupied strip from $L_{\overline{M}}$ with cardinality $|S'_o|=|S_e|$ and order its vertices, such that the first vertex $w\in S'_o\cap E_{\overline{M};p,q}$. Fill the empty strip $S_e$. Denote this new set by $A$, then
\[
|\partial_e A| = |\partial_e A'|+ |S'_o(I)| - |S'_o(E)| +  2
\]
and therefore $|\partial_e A| = |\partial_e D|- 2 |S'_o(E)| -1 < |\partial_e D|$.

\paragraph{\texttt{CASE (III)}}
For the general case we have  $s_e>1$ empty strips in the annulus $B_{D,Min} \setminus B_{D,max}$, denote them by $S_{e,1},\ldots,S_{e,s_e}$ where $s_e$. We will distinguish $q\geq 4$ and $q=3$. For  examples, see Figures \ref{fig:case3-1}, resp.~\ref{fig:case3-2}. We will first assume that there exists a strip  $S_{e,j}$ which is in some layer $L_k$ with $k<\overline{M}$ such that the vertices in the layers $L_{k-1}\cup L_{k+1}$ connected with $S_{e,j}$ are occupied.

\vspace{0.5cm}

    \begin{minipage}{0.45\textwidth}
        \centering
        \includegraphics[width=0.74\textwidth]{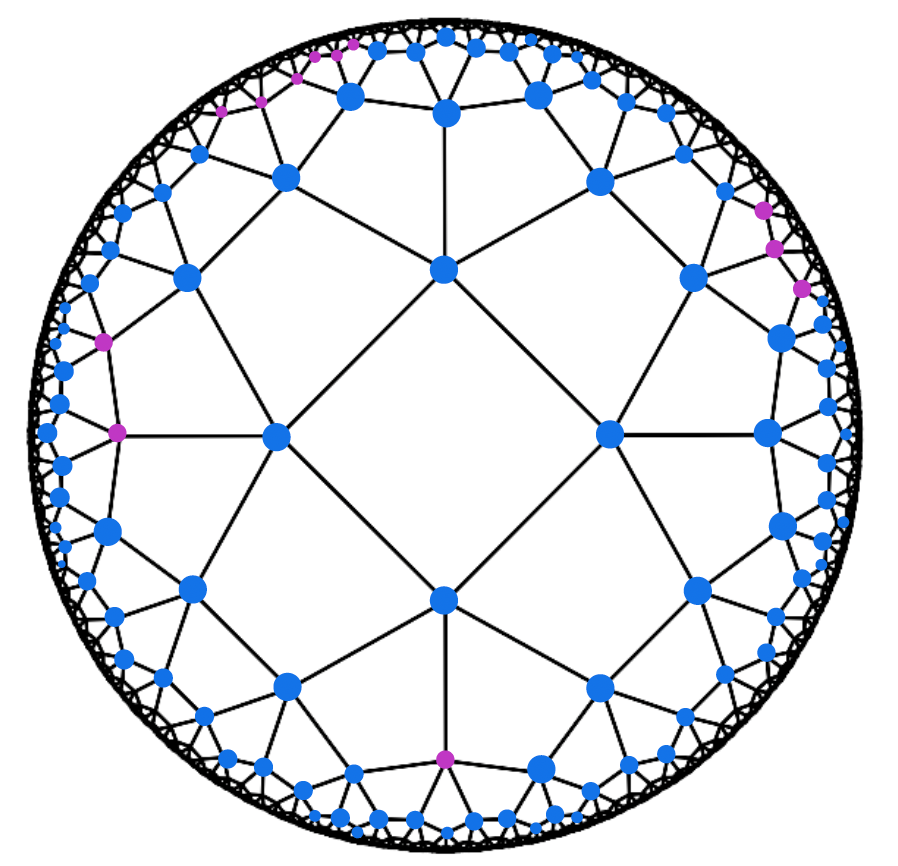} 
        \put(-86,66){$\origin$}
        \put(-83,71){ $\bullet$}
        \captionof{figure}{Example for a set $D$ (blue sites) for $p=4$ and $q\geq 4$. We highlighted the empty sites by pink vertices. We have two empty strips $S_{e,1}, S_{e,2}$ in $L_1$ and two empty strips $S_{e,3}, S_{e,4}$ in $L_2$.} \label{fig:case3-1}
    \end{minipage}
    \hfill
    \begin{minipage}{0.45\textwidth}
        \centering
        \includegraphics[width=0.74\textwidth]{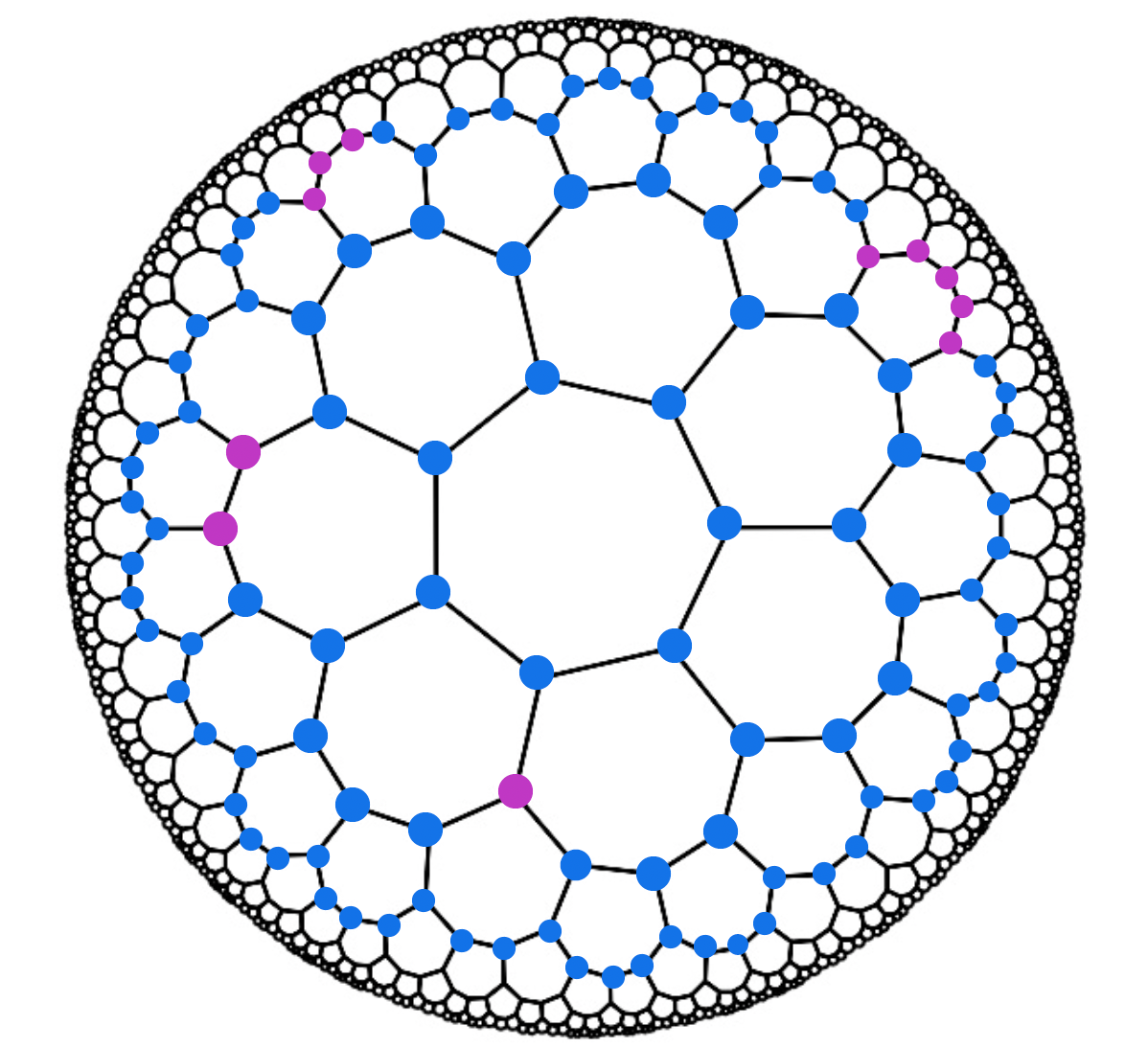} 
        \put(-86,62){$\origin$}
        \put(-83,68){ $\bullet$}
        \captionof{figure}{Example for a set $D$ (blue sites) for $p=7$ and $q=3$. We highlighted the empty sites by pink vertices. We have two empty strips $S_{e,1}, S_{e,2}$ in $L_1$ and two empty strips $S_{e,3}, S_{e,4}$ in $L_2$.}\label{fig:case3-2}
    \end{minipage}

    \begin{minipage}{0.45\textwidth}
        \centering
        \includegraphics[width=0.74\textwidth]{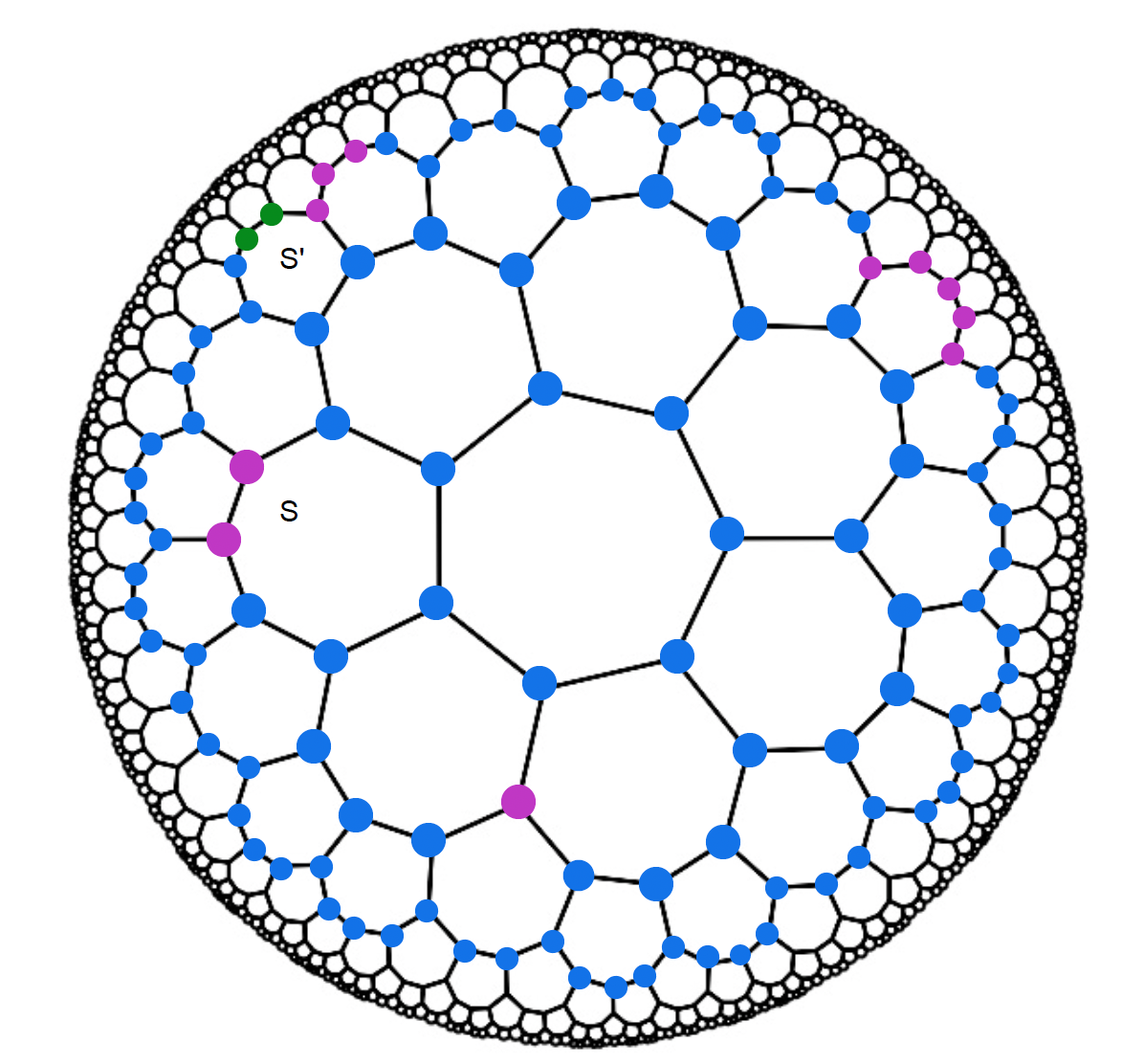} 
         \put(-86,62){$\origin$}
        \put(-83,68){ $\bullet$}
        \captionof{figure}{Example of a set $D$. We highlighted an empty strip $S$ and a strip $S'$ with the same cardinality with first vertex in $E_{2;7,3}$.} \label{fig:case3-3}
    \end{minipage}
    \hfill
    \begin{minipage}{0.45\textwidth}
        \centering
        \includegraphics[width=0.74\textwidth]{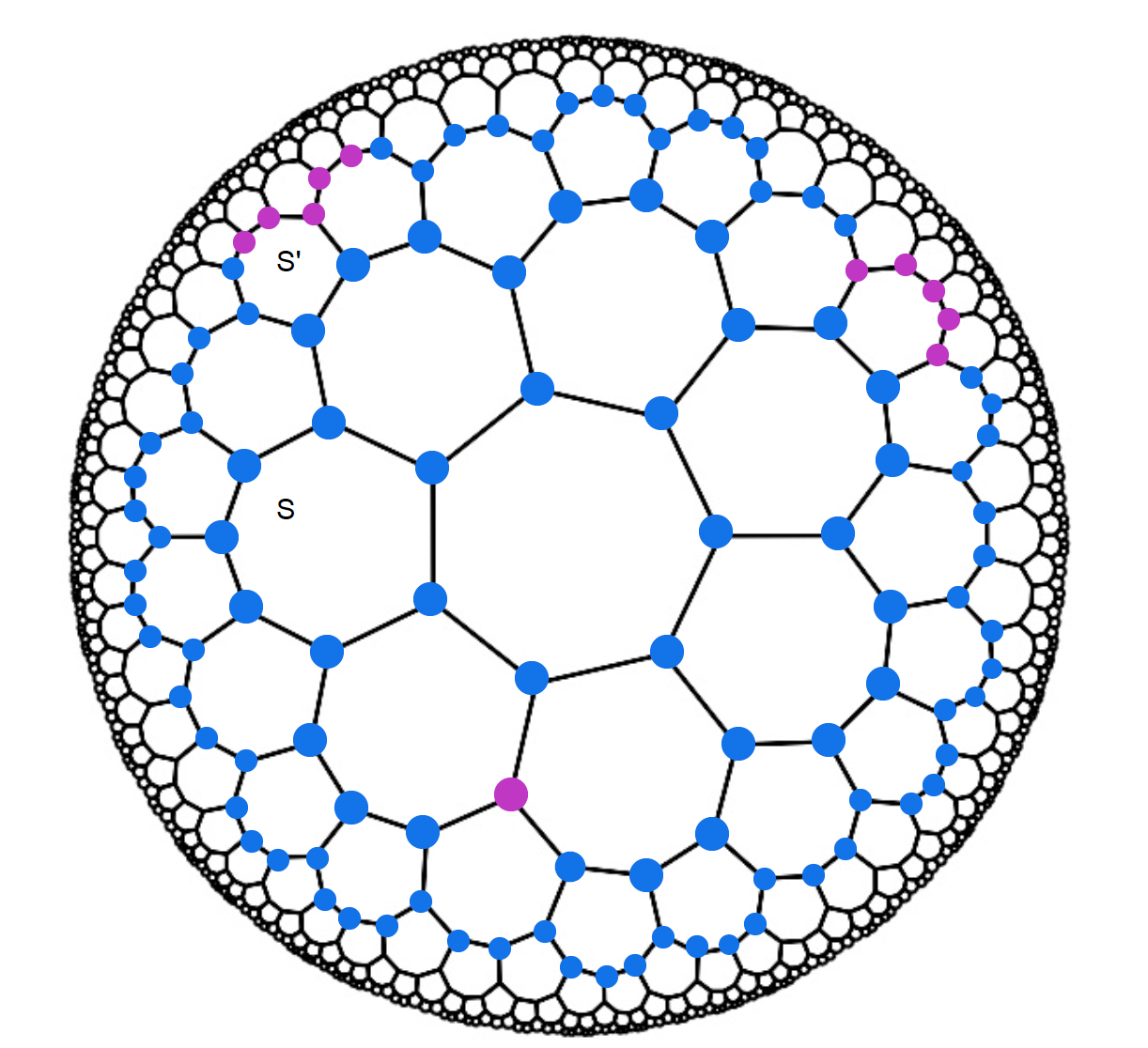} 
        \put(-86,62){$\origin$}
        \put(-83,68){ $\bullet$}
        \captionof{figure}{Example of a constructed set $A$ obtained from the set $D$ in Figure \ref{fig:case3-3}. We filled the empty strip $S$ with vertices from $S'$ with the same cardinality.}\label{fig:case3-4}
    \end{minipage}

\vspace{0.3cm}

Let us first consider $q\geq 4$. Consider any strip $S_{e,j} \ni v$. If we fill the empty vertex $v$, then
\begin{equation}\label{eq:case3}
| \partial_e D\cup\{v\}| \leq |\partial_e D| + (2-q),
\end{equation}
see also Figure \ref{fig:case3-1}. If there exists  a vertex $w\in E_{\overline{M};p,q} \cap \mathcal{N}_o$, then define $A\coloneqq (D\cup \{v\}) \setminus \{w\}$. The perimeter satisfies
$|\partial_e A| = |\partial_e D\cup \{v\}|+ (4-q).$

By Equation \eqref{eq:case3}, we obtain the claim since
$|\partial_e A| \leq |\partial_e D| +6-2q < |\partial_e D|.$
Otherwise if all occupied vertices $w\in L_{\overline{M}}\cap \mathcal{N}_o$ are in $I_{\overline{M};p,q}$, there exists a vertex in $w'\in E_{\overline{M}-1;p,q}\cap \mathcal{N}_o$. Analogously we obtain for $A\coloneqq (D\cup\{v\})\setminus \{w'\}$ that the perimeter can be bounded by
\[
|\partial_e A| \leq |\partial_e D| + 4-2q < |\partial_e D|.
\]

In the remaining argument let us assume that $q=3$. We have to distinguish between the case that in $L_{\overline{M}}\cap \mathcal{N}_o$ there are enough vertices to fill $S_{e,j}$ or not. Let us first consider the case that $|L_{\overline{M}}\cap \mathcal{N}_o| \geq  |S_{e,j}|$. Let $A'=D\cup S_{e,j}$, then the perimeter is equal to $|\partial_e A'|=|\partial_e D|-|S_{e,j}|-3$.
If the whole last layer is occupied, i.e. when $L_{\overline{M}}\cap \mathcal{N}_o = L_{\overline{M}}$, then define (and order) a strip in the last layer $S'_o =\{v_1,\ldots,v_K\}$ for $K=|S_{e,j}|$ such that $v_1\in E_{\overline{M};p,3}$. Set $A= (D\cup S_{e,j})\setminus S'_o$ and by a direct computation we get that $|\partial_e A| = |\partial_e D\cup S_{e,j}|+|S'_o(I)|-|S'_o(E)|+2$ and therefore
\[
|\partial_e A| \leq |\partial_e D|-2|S'_o(E)|-1< |\partial_e D|.
\]
Otherwise if there are empty vertices in $L_{\overline{M}}$, then define the strip $S'_o \subset L_{\overline{M}}$ with $v_1$ next to an empty vertex $w \in L_{\overline{M}}\cap \mathcal{N}_e$. 
Then we obtain for the perimeter, setting $A= (D\cup S_{e,j})\setminus S'_o$, then 
\[
|\partial_e A| \leq |\partial_e D\cup S_{e,j}| + |S'_o(I)|-|S'_o(E)| \leq |\partial_e D| - 2|S'_o(E)|-3 < |\partial_e D|.
\]
For examples of the sets $D$, resp. $A$, see Figures \ref{fig:case3-3} resp.~\ref{fig:case3-4}.


Let us now focus on the case that $|L_{\overline{M}}\cap \mathcal{N}_o| < |S_{e,j}|$, hence there are less occupied vertices than $|S_{e,j}|$. Call $h\coloneqq |L_{\overline{M}}\cap \mathcal{N}_o|$ and fill $h$ vertices $\{v_1,\dots,v_h\}$ in $S_{e,j}$. Then
\[
|\partial_e D\cup \{v_1,\ldots,v_h\}| = |\partial_e D| - h.
\]
Then we define $A$ from the previous set by removing all vertices from the last layer, i.e. $A\coloneqq (D\cup \{v_1,\ldots,v_h)\setminus (L_{\overline{M}}\cap \mathcal{N}_o)$. Then analogously
\[
|\partial_e A| \leq |\partial_e D| - 2|\mathcal{N}_o\cap E_{\overline{M};p,3}| -3 < |\partial_e D|.
\]

Assume now that a strip $S_{e,j}$ such that the occupied vertices in the layers $L_{k-1}\cup L_{k+1}$ are connected with $S_{e,j}$, does not exist.
Recall that we have $s_e$ empty (resp. occupied) strips in the annulus $B_{D,Min} \setminus B_{D,max}$. Denote by $S_{o,1},\ldots,S_{o,s_e}$ the occupied strips in layers  $L_j$ where $j\in \{\overline{m}+1,\ldots,\overline{M}\}$. Recall that $\bigcup_{i=1}^{s_e}S_{o,i} = \mathcal{N}_o$.
Then the number of occupied vertices in $\bigcup_{j=\overline{m}+1}^{\overline{M}}I_{j;p,3}$ satisfies
\[
t\coloneqq\left | \mathcal{N}_o \cap \left ( \bigcup_{j=\overline{m}+1}^{\overline{M}}I_{j;p,3}\right )\right | < o_{max} + s_e-1,
\]
since by assumption, $D\notin\mathcal{M}_N$. Then for $D\setminus \mathcal{N}_o = B_{D,max}$ we have that
\begin{align}\label{eq:the3.1q3}
|\partial_e B_{D,max}| &\leq |\partial_e D| +\sum_{j=1}^{s_e} \left (|S_{o,j}(I)| -1\right)- \sum_{j=1}^{s_e} \left (|S_{o,j}(E)|+1\right )\notag\\
& = |\partial_e D| +(t-s_e) -(|\mathcal{N}_o|-t+s_e).
\end{align}

If $|\mathcal{N}_o| \leq |L_{\overline{m}+1}|$, we define $A\coloneqq B_{D,max}\cup S'$ where $|S'|=|\mathcal{N}_o|$ and such that $S'$ contains $o_{max}$ vertices in $I_{\overline{m}+1;p,3}$.
Otherwise there exists $k\in \{\overline{m}+1,\ldots,\overline{M}-1\}$ such that
\[
\sum_{j=\overline{m}+1}^k|L_{j}| < |\mathcal{N}_o| \leq \sum_{j=\overline{m}+1}^{k+1} |L_j|.
\]
In this case we define $A$ by $A\coloneqq B_{D,max}\cup \bigcup_{j=1}^k L_j\cup S'$, where $S'$ is a strip of length $|\mathcal{N}_o| - \sum_{j=\overline{m}+1}^k|L_{j}|$ and such that $S'$ contains the maximal number of vertices in $I_{k+1;p,3}$. For an example see Figure \ref{fig:ex_S}.
Then using Equation \eqref{eq:the3.1q3}
   \begin{align}
       |\partial_e A|  &\leq |\partial_e D| + (t-s_e-o_{max}+1)< |\partial_e D|
    \end{align}
    since $t< o_{max}+s_e-1$.

\begin{figure}[!hbtp]
\centering
\includegraphics[scale=0.3]{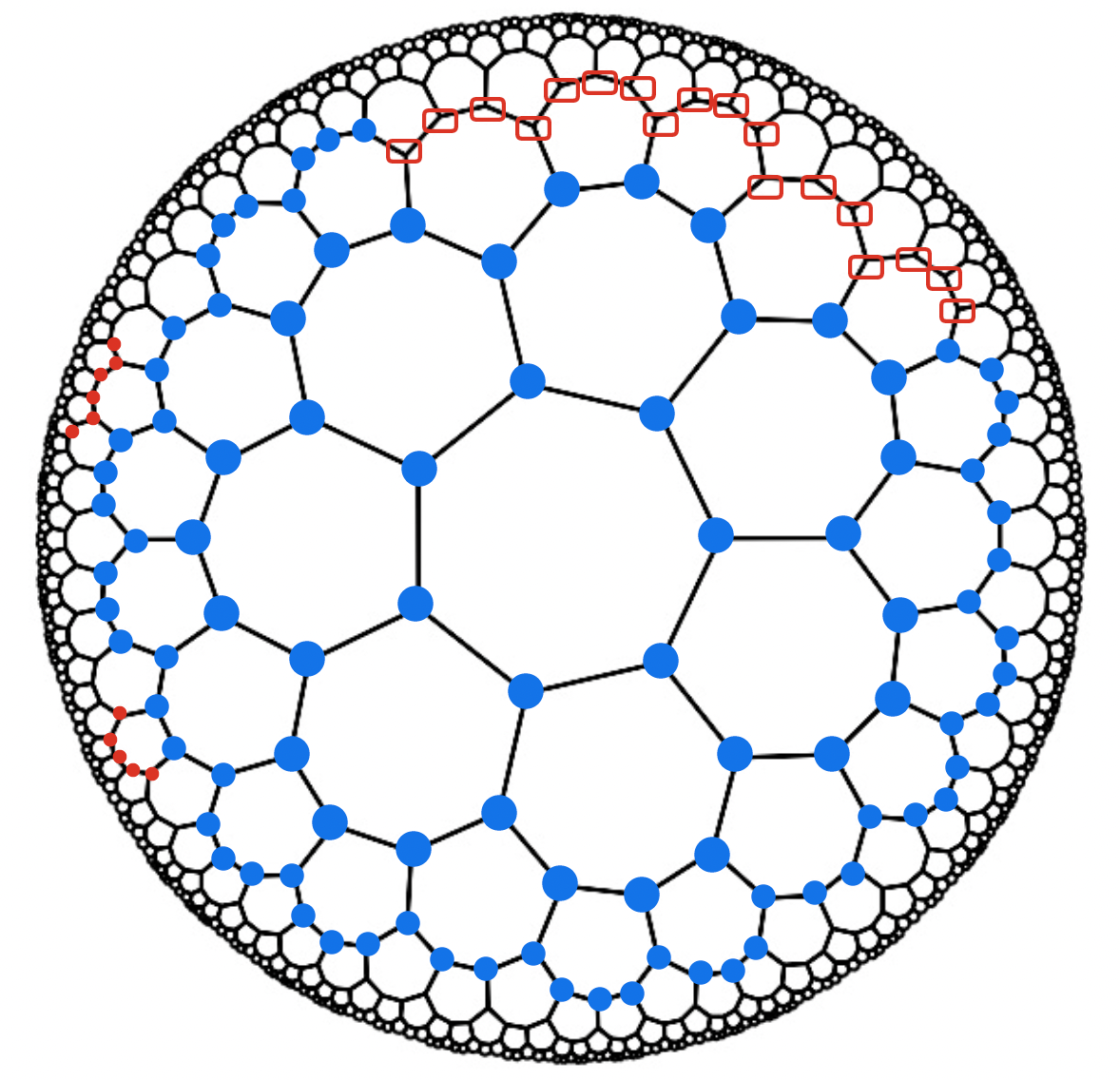}
\put(-96,90){$\origin$}
\put(-93,84){ $\bullet$}
\caption{Example of a set $D$ with two stripes of red occupied vertices, $S_{o,1}, S_{o,2}$ in the last layer and one empty stripe of length 18 in the previous layer which will be occupied by $|S_{o,1} \cup S_{o,2}|=11$ many vertices.}\label{fig:ex_S}
\end{figure}

\newpage
\subsection{Proof of Theorem \ref{thm:inf}}

Remark that for a fixed $N\in \mathbb{N}$ and let $Q \in \mathcal{S}_N, \, \mathscr{M} \in \mathcal{M}_N$. By Theorem \ref{minimizers} we have that
\begin{align}
\frac{|\partial_e Q|}{|Q|}  \geq\frac{|\partial_e \mathscr{M}|}{|\mathscr{M}|} 
&\geq \inf \left  \{ \frac{|\partial_e A|}{|A|} : 0<|A|<\infty\right \}. \notag
\end{align}
On the other hand by Theorem 4.1 in \cite{haggstrom2002explicit} we have that
\[
i_e(\Lpq) = \inf \left \{ \frac{|\partial_e A|}{|A|} : 0<|A|<\infty\right \} =   (q-2) \sqrt{1-\frac{4}{(p-2)(q-2)}}.
\]
Thus, it follows trivially that
\begin{align}
\lim_{N \to \infty}\frac{|\partial_e \mathscr{M}|}{|\mathscr{M}|} \geq i_e(\Lpq).
\end{align}
The claim of the result will follow from Lemmas \ref{lem:balls} and \ref{upper} which we will prove below by transfer matrix methods.
\begin{lem}\label{lem:balls}
We have that 
\[
\lim_{n \to \infty} \frac{|\partial_e B_n(\origin)|}{|B_n(\origin)|} = i_e(\mathcal{L}_{p,q}).
\]
\end{lem}
\begin{lem}\label{upper} 
For $\mathscr{M} \in  \mathcal{M}_N$ we have that
\begin{align}
    \lim_{|\mathscr{M}|\rightarrow \infty} \frac{|\partial_e \mathscr{M}|}{|\mathscr{M}|} \leq i_e(\Lpq).
\end{align}

\end{lem}
\vspace{0.2cm}
\subsubsection{Proof of Lemma \ref{lem:balls}}

Observe that the perimeter of a regular ball of radius $n$ is equal to (depending on $p\geq 4$ resp. $p=3$)
\[
|\partial_e B_n(\origin)| = |I_{n+1;p,q}|, \, \, \text{ resp. } \, \, |\partial_e B_n(\origin)| = 2|I''_{n+1;3,q}|+|I'_{n+1;3,q}|,
\]
so that the ratio becomes
\[
\frac{|\partial_e B_n(\origin)|}{|B_n(\origin)|}
 = \begin{cases}
\frac{|I_{n+1;p,q}|}{\sum_{k=0}^n \left (|I_{k;p,q}|+|E_{k;p,q}| \right )} & \text{ if } p\geq 4, \\
& \\
\frac{2|I''_{n+1;3,q}|+ |I'_{n;p,q}|}{3+\sum_{k=1}^n  \left (|I'_{k;3,q}|+|I''_{k;3,q}| \right )} & \text{ if } p=3.
 \end{cases}
 \]

From \cite{RNO} we have the following recursion relation for $I_{n;p,q}, E_{n;p,q}$, see also \cite[Equations (2.14) and (2.16)]{RNO}, for $p\geq 4$ and $n\geq 0$
\begin{align}
\binom{I_{n+1;p,q}}{E_{n+1;p,q}} = T_1 \binom{I_{n;p,q}}{E_{n;p,q}}
\end{align}
where $(I_{0;p,q},E_{0;p,q})=(0,p)$ and

\begin{align}\label{eq:T1}
T_1 = \begin{pmatrix}
q-3 & q-2 \\
8 -3p-3q+pq & 5-2p-3q+pq
\end{pmatrix}
\end{align}
resp. for $p=3$, $n\geq 1$

\begin{align}
\binom{I''_{n+1;3,q}}{I'_{n+1;3,q}} = T_2 \binom{I''_{n;3,q}}{I'_{n;3,q}}
\end{align}
where $(I''_{1;3,q},I'_{1;3,q})=(3,3(q-4))$ and

\begin{align}\label{eq:T2}
T_2 = \begin{pmatrix}
1 & 1 \\
q-6& q-5
\end{pmatrix}.
\end{align}
We will distinguish between $p\geq 4$ and $p=3$. Let us consider the first case. 
\paragraph{\texttt{CASE (I), $p\geq 4$.}}
Due to the recursion relation, we have that
\[
|I_{n+1;p,q}| = (1,0)T_1^{n+1} \begin{pmatrix}0 \\ p \end{pmatrix}.
\]
Diagonalizing $T_1$ we obtain that
\[
 T^n_1 = (e_+, e_-)\begin{pmatrix} \lambda^n_+ & 0 \\ 0 & \lambda^n_-\end{pmatrix} (e_+, e_-)^{-1},
\]
where the eigenvalues are equal to
\begin{align}
\lambda_{\pm} &= \frac{1}{2} \left (2+p(q-2)-2q \pm \sqrt{(p-2)(q-2)(q(p-2)-2p)}\right ) \notag \\
&= \left ( \frac{1}{2}(p-2)(q-2)-1\right ) \left ( 1 \pm \sqrt{1-\frac{4}{((p-2)(q-2)-2)^2} }\right ),
\end{align}
and the corresponding eigenvectors
\[
e_+ = \left(\frac{2}{p-4+\frac{\sqrt{(p-2)(p(q-2)-2q)}}{\sqrt{q-2}}} ,1\right)^T, \, \, \text{ resp. } e_- = \left(\frac{2}{p-4-\frac{\sqrt{(p-2)(p(q-2)-2q)}}{\sqrt{q-2}}} ,1\right)^T.
\]

Therefore,
\begin{equation}
|I_{n+1;p,q} |= \frac{p\sqrt{q-2} \left ( \lambda^{n+1}_+ - \lambda^{n+1}_-\right )}{\sqrt{(p-2)(p(q-2)-2q)}} 
\end{equation}
resp.

\begin{align}
\sum_{k=0}^{n} \left (|I_{k;p,q}|+|E_{k;p,q}| \right ) &= \sum_{k=0}^n (1,1)T^k_1 \binom{0}{p}=\frac{p}{2\sqrt{(p-2)(p(q-2)-2q)}} \sum_{k=0}^n \left ( a_-\lambda^k_{-} + a_+ \lambda^k_+\right ) \notag
\end{align}
with
\begin{align}\label{eq:apm}
a_{\pm} &=(\sqrt{(p-2)(p(q-2)-2q)} \mp 2 \sqrt{q-2} \pm p\sqrt{q-2})=(p-2)\sqrt{q-2} \left (\pm1+\sqrt{1-\frac{4}{(p-2)(q-2)}}  \right).
\end{align}
Then we can conclude that
\begin{align}\label{limit}
    \lim_{n \to \infty} \frac{|\partial_e (B_n(\origin)|}{|B_n(\origin)|}
    &=\lim_{n \to \infty}  2 \sqrt{q-2} \frac{\left ( \lambda^{n+1}_+ - \lambda^{n+1}_-\right )}{\sum_{k=0}^n \left ( a_-\lambda^k_{-} + a_+ \lambda^k_+\right )}
    =\lim_{n \to \infty}  2 \sqrt{q-2} \frac{\left ( \lambda^{n+1}_+ - \lambda^{n+1}_-\right )}{a_{-} \frac{1-\lambda_{-}^{n+1}}{1-\lambda_{-}}+a_{+} \frac{1-\lambda_{+}^{n+1}}{1-\lambda_{+}}} \notag \\
    &=\lim_{n \to \infty} 2 \sqrt{q-2} \frac{1-\left (\frac{\lambda_{-}}{\lambda_{+}} \right )^{n+1}}{ \frac{a_{-}}{1-\lambda_{-}}\frac{1-\lambda_{-}^{n+1}}{\lambda_{+}^{n+1}}+ \frac{a_{+}}{1-\lambda_{+}}\frac{1-\lambda_{+}^{n+1}}{\lambda_{+}^{n+1}}}
    =2\sqrt{q-2}\frac{\lambda_{+}-1}{a_{+}} 
\end{align}
since $ \left (\frac{\lambda_{-}}{\lambda_{+}} \right)^{n} \to 0$ for $n \to \infty$. Moreover, by a direct computation we can compute that
\[
i_e(\Lpq) = 2\sqrt{q-2}\frac{\lambda_{+}-1}{a_{+}}, 
\]
which was defined in Equation \eqref{eq:iso_Lpq}. 

\paragraph{\texttt{CASE (II), $p=3$.}}
We diagonalize $T_2$ by computing its eigenvalues resp.~eigenvectors:

\begin{align}\label{eq:eigen}
    \lambda_\pm 
    & = \frac{1}{2}\left ( q-4 \pm \sqrt{(q-6)(q-2)}\right )
\end{align}
resp.
\begin{align}
    e_+= \left ( \frac{6-q+\sqrt{12-8q+q^2}}{2(q-6)}, 1\right )^T, \,\, \text{ resp. } e_-=\left ( \frac{6-q-\sqrt{12-8q+q^2}}{2(q-6)}, 1\right )^T.
\end{align}
Analogously to before we write for the fraction
\begin{equation}
 \frac{|\partial_e B_n(\origin)|}{|B_n(\origin)|} =
\frac{2(1,0)T_2^{n} \begin{pmatrix} 3 \\ 3(q-4) \end{pmatrix}+(0,1)T_2^{n} \begin{pmatrix} 3 \\ 3(q-4) \end{pmatrix}}{3+\sum_{k=0}^{n-1} (1,1)T_2^{k} \begin{pmatrix}3 \\ 3(q-4) \end{pmatrix}}
\end{equation}
which reduces to 
\begin{align}
|I''_{n+1;3,q} |&= \frac{3 \left(\left(\sqrt{(q-6)(q-2)}-(q-2)\right) \lambda_-^{n}+\left(\sqrt{(q-6)(q-2)}+q-2\right) \lambda_+^{n}\right)}{2 \sqrt{(q-6)(q-2)}}, \notag 
\end{align}
and
\begin{align}
2|I''_{n+1;3,q}| + |I'_{n+1;3,q}| & = \frac{3 (q-2) \left(\left(4-q+\sqrt{(q-6) (q-2)}\right) \lambda_-^{n}+\left(q-4+\sqrt{(q-6) (q-2)}\right) \lambda_+^{n}\right)}{2 \sqrt{(q-6) (q-2)}} \notag 
\end{align}
resp.
\begin{align}
|B_n(\origin)| &=3+\sum_{k=0}^{n-1}\left (a_+\lambda_+^k+a_-\lambda^k_- \right )=3+a_+\frac{1-\lambda_+^{n}}{1-\lambda_+}+a_-\frac{1-\lambda_-^{n}}{1-\lambda_-},
\end{align}
where 
\begin{align}
    a_\pm=\frac{3}{2} \left ( q-3 \pm \frac{(q-5)(q-2)}{\sqrt{(q-6)(q-2)}} \right ). \notag 
\end{align}
By a direct computation, we obtain the claim
\begin{align}\label{limit_p3}
    \lim_{n \to \infty} \frac{|\partial_e B_n(\origin)|}{|B_n(\origin)|} 
    &=\frac{3}{2}\frac{q-2}{\sqrt{(q-6)(q-2)}}\left ( q-4+\sqrt{(q-6)(q-2)}\right ) \frac{\lambda_+-1}{a_+} \notag \\
    &=3\sqrt{\frac{q-2}{q-6}} \frac{\lambda_+(\lambda_+-1)}{a_+}
    =i_e (\mathcal{L}_{3,q})
\end{align}
since $ \left (\frac{\lambda_{-}}{\lambda_{+}} \right)^{n} \to 0$ for $n \to \infty$.

\vspace{0.2cm}
\subsubsection{Proof of Lemma \ref{upper}}
In the following let us give an exact formula for the perimeter of $\mathscr{M} \in \mathcal{M}_N$ for all $N \in \mathbb{N}$. First, we observe that if $N$ is such that $\mathcal{N}_o=\emptyset$, then $\mathscr{M} \equiv B_n(\origin)$ and we conclude by applying Lemma \ref{lem:balls}. Thus, suppose $\mathcal{N}_o\neq \emptyset$, i.e. $\mathscr{M} \not \equiv B_n(\origin)$, then
\begin{align}\label{estimate_perimeter_min}
|\partial_e \mathscr{M}|   
=\begin{cases}
         |I_{n+1;p,q}|+(q-4)(|\mathcal{N}_o \cap I_{n+1;p,q}|-1)+(q-2)(|\mathcal{N}_o \cap E_{n+1;p,q}|+1) 
         \qquad &\text{ if } p\geq 4, \\
         2|I''_{n+1;3,q}|+|I'_{n+1;3,q}|+(q-6)(|\mathcal{N}_o \cap I''_{n+1;3,q}|-1)+(q-4)(|\mathcal{N}_o \cap I'_{n+1;3,q}|+1) \qquad &\text{ if } p=3.
\end{cases}
\end{align}
For an example see Figure \ref{fig:ex_edges}.
\begin{figure}[!hbtp]
\centering
\includegraphics[scale=0.3]{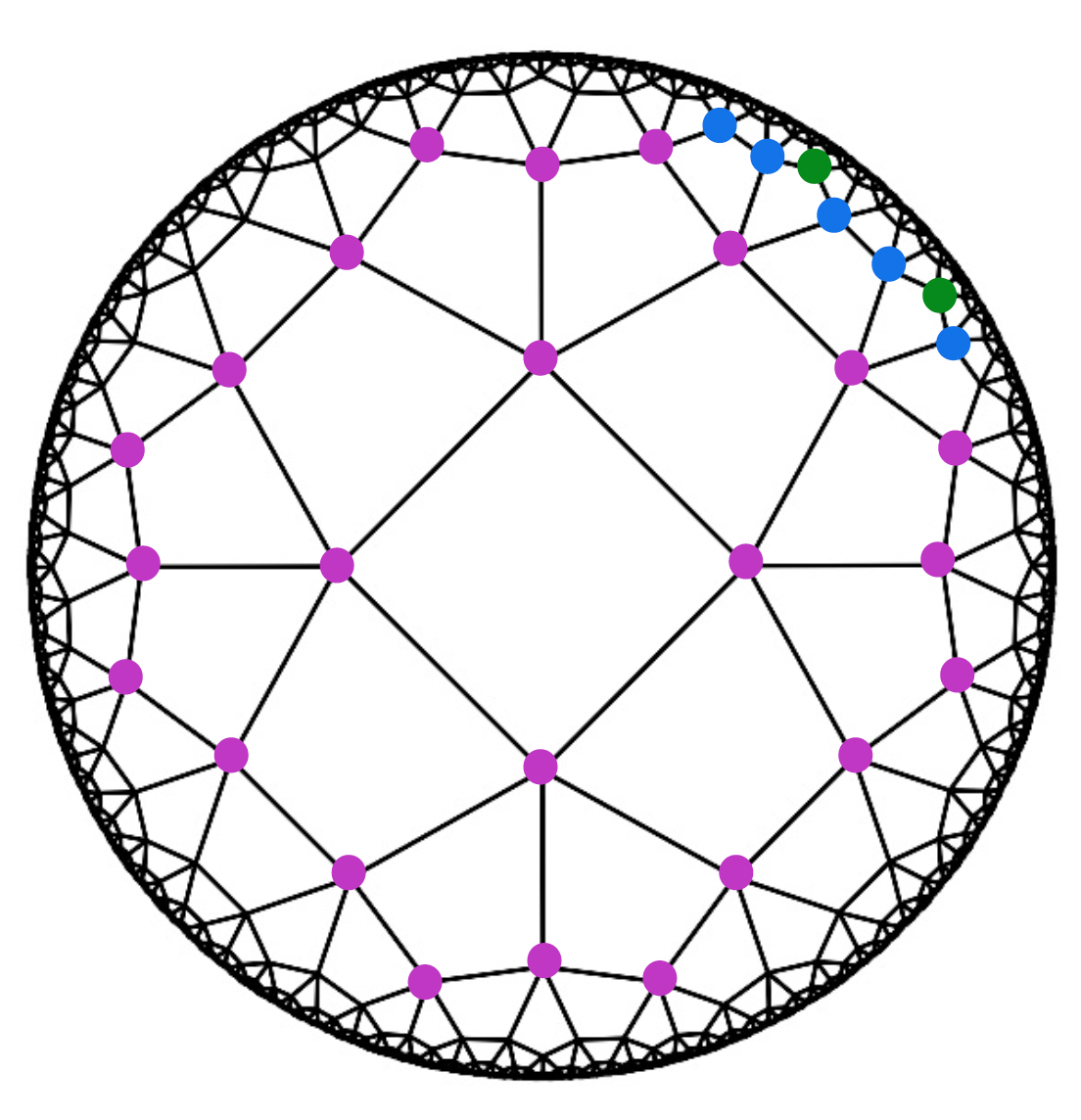}
\put(-89,78){$\origin$}
\put(-87,84){ $\bullet$}
\caption{Example of a set $\mathscr{M}$ satisfying the relation \eqref{estimate_perimeter_min} for $p=4$, $q=5$. We have that $|\mathcal{N}_o|=7$, $|\mathcal{N}_o \cap I_{2;4,5}|=5$, $|\mathcal{N}_o\cap E_{2;4,5}|=2$, $|I_{2;4,5}|=48$ and $|\partial_e \mathscr{M}|=61= 48 + 1\cdot (5-1)+3\cdot (2+1)$} \label{fig:ex_edges}
\end{figure}

We distinguish two cases: $p\geq 4$ and  $p=3$.

\paragraph{\texttt{CASE (I), $p \geq 4$.}}
We note that $|\mathcal{N}_o \cap I_{n+1;p,q}| \geq\frac{|I_{n+1;p,q}|}{|L_{n+1}|}|\mathcal{N}_o|$, see Remark \ref{bound1} for details.
By using Equation \eqref{estimate_perimeter_min}, we can bound from above the ratio as follows
\begin{align}
    \frac{|\partial_e \mathscr{M}|}{|\mathscr{M}|}&=\frac{|I_{n+1;p,q}|+(q-4)(|\mathcal{N}_o \cap I_{n+1;p,q}|-1)+(q-2)(|\mathcal{N}_o \cap E_{n+1;p,q}|+1) }{\left | \bigcup_{j=0}^{n}L_j\right|+|\mathcal{N}_o|} \notag \\
    &=\frac{|I_{n+1;p,q}|+(q-2)|\mathcal{N}_o|-2|\mathcal{N}_o \cap I_{n+1;p,q}|+2 }{\left | \bigcup_{j=0}^{n}L_j\right|+|\mathcal{N}_o|} \notag \\
& \leq \frac{|I_{n+1;p,q}|+ \left (q-2-2\frac{|I_{n+1}|}{|L_{n+1}|} \right )|\mathcal{N}_o|+2 }{\left | \bigcup_{j=0}^{n}L_j\right|+|\mathcal{N}_o|}.
\end{align}
Since $| \bigcup_{j=0}^{n}L_j|>0$ and $\left (q-2-2\frac{|I_{n+1}|}{|L_{n+1}|} \right )>1$, see Remark \ref{bound2} for details, we have an increasing function of $|\mathcal{N}_o| < |L_{n+1}|$ and we obtain
\begin{align}
   \frac{|\partial_e \mathscr{M}|}{|\mathscr{M}|}& < \frac{|I_{n+1;p,q}|+\left (q-2-2\frac{|I_{n+1}|}{|L_{n+1}|} \right ) |L_{n+1}|+2 }{\left | \bigcup_{j=0}^{n}L_j\right|+|L_{n+1}|} \notag \\
   &= \frac{|I_{n+1;p,q}|+\left (q-2 \right )(|I_{n+1;p,q}|+|E_{n+1;p,q}|)-2|I_{n+1;p,q}|+2 }{\left | \bigcup_{j=0}^{n+1}L_j\right|} \notag \\
   &= \frac{\overbrace{(q-3)|I_{n+1;p,q}|+(q-2)|E_{n+1;p,q}|}^{|I_{n+2;p,q}|}+2 }{\left | \bigcup_{j=0}^{n+1}L_j\right|} \notag \\
   &= \frac{|I_{n+2;p,q}|+2}{\left | \bigcup_{j=0}^{n+1}L_j\right|} 
   \underset{n \to \infty}{\to} i( \Lpq).
\end{align}

\paragraph{\texttt{CASE (II), $p=3$.}} First, we note that, if $p=3$, then $q \geq 7$ by the assumption $1/p+1/q<1/2$. Moreover, we observe that $|\mathcal{N}_o \cap I''_{n+1;p,q}| \geq\frac{|I''_{n+1;p,q}|}{|L_{n+1}|}|\mathcal{N}_o|$, thus
\begin{align}
    \frac{|\partial_e\mathscr{M}|}{|\mathscr{M}|}&=\frac{2|I''_{n+1;p,q}|+|I'_{n+1;p,q}|+(q-6)(|\mathcal{N}_o \cap I''_{n+1;p,q}|-1)+(q-4)(|\mathcal{N}_o \cap I'_{n+1;p,q}|+1)}{|\bigcup_{j=0}^{n} L_j| +|\mathcal{N}_o|} \notag \\
    &=\frac{2|I''_{n+1;p,q}|+|I'_{n+1;p,q}|+(q-4)|\mathcal{N}_o|-2|\mathcal{N}_o \cap I''_{n+1;p,q}|+2}{|\bigcup_{j=0}^{n} L_j| +|\mathcal{N}_o|} \notag \\
    & \leq \frac{2|I''_{n+1;p,q}|+|I'_{n+1;p,q}|+\left (q-4 - 2\frac{|I''_{n+1;p,q}|}{|L_{n+1}|} \right )|\mathcal{N}_o|+2}{|\bigcup_{j=0}^{n} L_j| +|\mathcal{N}_o|}.
\end{align}
Since $| \bigcup_{j=0}^{n}L_j|>0$ and $\left (q-4 - 2\frac{|I''_{n+1;p,q}|}{|L_{n+1}|} \right )>1$,  see Remark \ref{bound2} for details, we again have an increasing function of $|\mathcal{N}_o| < |L_{n+1}|$ and we obtain
\begin{align}
   \frac{|\partial_e \mathscr{M}|}{|\mathscr{M}|}& <  \frac{|I'_{n+1;p,q}|+(q-4)|L_{n+1}|+2}{|\bigcup_{j=0}^{n+1} L_j|}=\frac{(q-3)|I'_{n+1;p,q}|+(q-4)|I''_{n+1;p,q}|+2}{|\bigcup_{j=0}^{n+1} L_j|} \notag \\
   &=\frac{\overbrace{(q-5)|I'_{n+1;p,q}|+(q-6)|I''_{n+1;p,q}|}^{I'_{n+2;p,q}}+ \overbrace{2|I'_{n+1;p,q}|+2|I''_{n+1;p,q}|}^{2I''_{n+2;p,q}}+2}{|\bigcup_{j=0}^{n+1} L_j|}\underset{n \to \infty}{\to} i(\Lpq)
   \end{align}

\begin{remark}\label{bound1}
    By Definition \ref{def:minset}, we recall that $|\mathcal{N}_o|$ contains the maximal number of vertices in $I_{n+1;p,q}$, i.e., $|\mathcal{N}_o \cap I_{n+1;p,q}|$ is maximal. For $j=1,...,|L_{n+1}|$, we define $S_j$ the ordered strips $L_{n+1}$ of length $|S_j|=|\mathcal{N}_o|$. Thus, we obtain 
    \begin{align}
        |I_{n+1;p,q}|=\sum_{j=1}^{|L_{n+1}|}\frac{|S_j \cap I_{n+1;p,q}|}{|S_j|}\leq |L_{n+1}|\max_{j\in \{1,...,|L_{n+1}|\}} \frac{|S_j \cap I_{n+1;p,q}|}{|S_j|}=|L_{n+1}|\frac{|\mathcal{N}_o \cap I_{n+1;p,q}|}{|\mathcal{N}_o|}
    \end{align}
\end{remark}
\begin{remark}\label{bound2} We prove that $q-2-2\frac{|I_{n+1;p,q}|}{|L_{n+1}|}>1$ and $q-4 - 2\frac{|I''_{n+1;p,q}|}{|L_{n+1}|}>1$. By definition of $I_{n+1;p,q}$ and recall that the eigenvalues $\lambda_{\pm}$ were defined in Equation \eqref{eq:eigen} and the constants $a_{\pm}$ in Equation \eqref{eq:apm},
    \begin{align}
        q-2-2\frac{|I_{n+1;p,q}|}{|L_{n+1}|}
        &=q-2-4 \frac{\sqrt{q-2} (\lambda_+^{n+1}-\lambda_-^{n+1})}{a_+\lambda_+^{n+1}+a_-\lambda_-^{n+1}} \geq q-2-4 \frac{\sqrt{q-2}}{a_+} \notag \\
        &=q-2-\frac{4}{(p-2) \left ( 1 +\sqrt{1-\frac{4}{(p-2)(q-2)}}\right )}>1
    \end{align}
    where we used that the first function is decreasing in $n$,  the second function is increasing in $p$ and $q$ with $1/p+1/q<1/2$. Moreover,
    \begin{align}
        g(q,n)&\coloneqq q-4 - 2\frac{|I''_{n+1;p,q}|}{|L_{n+1}|} 
        =q-4-\frac{3 \left(\left(1-\sqrt{\frac{q-2}{q-6}}\right) \lambda_-^{n+1}+\left(1+\sqrt{\frac{q-2}{q-6}}\right) \lambda_+^{n+1}\right)}{a_+\lambda_+^{n+1}+a_-\lambda_-^{n+1}} \notag \\
        & \geq q-4-\frac{3 \left(\left(1-\sqrt{\frac{q-2}{q-6}}\right) \lambda_-+\left(1+\sqrt{\frac{q-2}{q-6}}\right) \lambda_+\right)}{a_+\lambda_++a_-\lambda_-} \geq g(7,1)>1
    \end{align}
    where we used that the first function is increasing in $n$ and the second function is increasing in $q$, where $q\geq 7$.
\end{remark}

\vspace{1cm}
\subsection*{Acknowledgments} 
The authors are grateful to Russ Lyons for precious comments on a previous version of this manuscript.
The authors would like to thank Annika Brockhaus and the authors of the Python library \cite{tiles} for helping producing the pictures, David Adame-Carillo for interesting discussion on the partitioning of the strips with maximum cardinality and Norbert Peyerimhoff for explaining his paper~\cite{keller2008geometric} to them. V.J. thanks GNAMPA. M.D'A.~is grateful to the Department of Mathematics of the University of Utrecht for excellent working conditions in the occasion of an invitation (November 2024), during which this work has been partly done.

\vspace{1cm}
\printcredits

\bibliographystyle{cas-model2-names}
\bibliography{biblio}

\end{document}